\documentclass[a4paper, fleqn, abstract=true, headings=small, english, bibliography=totoc]{scrartcl}

\usepackage{setspace}
\usepackage{indentfirst} 

\usepackage[T1]{fontenc}
\usepackage[utf8]{inputenc}

\usepackage{amsfonts,amsmath,amssymb,esint}
\usepackage{mathtools}
\usepackage{amsthm}
\usepackage{thmtools}

\usepackage{mathrsfs}
\usepackage{stmaryrd}

\usepackage[dvipsnames]{xcolor}

\usepackage{tikz-cd}

\usepackage{graphicx}

\usepackage{array}

\usepackage[alwaysadjust]{paralist}
\usepackage{enumitem}

\usepackage{babel}
\usepackage[theoremfont, osf]{newpxtext}
\usepackage[slantedGreek]{newpxmath}

\theoremstyle{plain}
\newtheorem{proposition}{Proposition}[section]
\newtheorem{lemma}[proposition]{Lemma}

\newtheorem{theorem}[proposition]{Theorem}
\newtheorem{corollary}[proposition]{Corollary}
\newtheorem{definition}[proposition]{Definition}

\newtheorem{openproblem}[proposition]{Open problem}

\usepackage{etoolbox}
\theoremstyle{definition}

\AtEndEnvironment{example}{\quad\null\hfill\qedsymbol}%

\theoremstyle{remark}

\AtEndEnvironment{remark}{\quad\null\hfill\qedsymbol}%

\declaretheoremstyle[
spaceabove=.3\topsep, 
spacebelow=.3\topsep,
headfont=\small\itshape,
headpunct={. ---},
notefont=\normalfont\itshape, 
bodyfont=\itshape,
headindent=\parindent,
numbered=yes,
]{mystyle1}
\declaretheorem[style=mystyle1, name=Step]{Step}
\newcommand{\resetstep}{\setcounter{Step}{0}}

\usepackage[shortcuts]{extdash}

\definecolor{purple}{cmyk}{0.75,0.90,0,0}
\definecolor{DB}{rgb}{0.07,0.0,0.5}
\definecolor{DG}{rgb}{0.0,0.37,0.07}
\definecolor{DR}{rgb}{0.37,0,0.07}

\usepackage[
pdftitle={Improved class R},%
pdfauthor={Antoine Detaille}, hyperindex=true, colorlinks=true, urlcolor={DR}, linkcolor={DR}, menucolor={DR}, citecolor={DR}, anchorcolor={DR},linktoc=all, pagebackref, final]{hyperref}

\newcommand{\dd}{\mathop{}\!\diffd}

\newcommand{\diffd}{{\operatorfont d}}

\newcommand{\bB}{\varmathbb{B}}

\newcommand{\bN}{\varmathbb{N}}

\newcommand{\bR}{\varmathbb{R}}
\newcommand{\bS}{\varmathbb{S}}
\newcommand{\bT}{\varmathbb{T}}

\newcommand{\bZ}{\varmathbb{Z}}

\newcommand{\mA}{\mathscr{A}}

\newcommand{\mE}{\mathscr{E}}

\newcommand{\mH}{\mathscr{H}}

\newcommand{\mK}{\mathscr{K}}
\newcommand{\mL}{\mathscr{L}}
\newcommand{\mM}{\mathscr{M}}
\newcommand{\mN}{\mathscr{N}}

\newcommand{\mP}{\mathscr{P}}

\newcommand{\mS}{\mathscr{S}}
\newcommand{\mT}{\mathscr{T}}
\newcommand{\mU}{\mathscr{U}}
\newcommand{\mV}{\mathscr{V}}
\newcommand{\mW}{\mathscr{W}}

\newcommand{\st}{\mathpunct{:}}

\newcommand{\ulrho}{\underline{\rho}}
\newcommand{\olrho}{\overline{\rho}}

\newcommand{\sm}{\mathrm{sm}}

\newcommand{\sh}{\mathrm{sh}}

\newcommand{\Rclas}{\mathscr{R}^{\textnormal{cros}}}
\newcommand{\Rrig}{\mathscr{R}^{\textnormal{rig}}}
\newcommand{\Rsmooth}{\mathscr{R}^{\textnormal{uncr}}}
\newcommand{\trun}{_\mathrm{tr}}
\newcommand{\rtop}{\mathrm{top}}

\newcommand{\ver}{\mathrm{v}}
\newcommand{\hor}{\mathrm{h}}

\newcommand{\restrfun}[1]{_{\vert {#1}}}

\DeclareMathOperator{\id}{id}

\DeclareMathOperator{\dist}{dist}

\DeclarePairedDelimiter{\abs}{\lvert}{\rvert}
\DeclarePairedDelimiter{\bracks}{\lbrack}{\rbrack}
\DeclarePairedDelimiter{\norm}{\lVert}{\rVert}
\DeclarePairedDelimiter{\parens}{\lparen}{\rparen}
\DeclarePairedDelimiter{\set}{\lbrace}{\rbrace}
\DeclarePairedDelimiter{\floor}{\lfloor}{\rfloor}

\numberwithin{equation}{section}

\begin{document}

\title{An improved dense class in Sobolev spaces to manifolds}
\author{Antoine Detaille\footnote{Universite Claude Bernard Lyon 1, CNRS, Centrale Lyon, INSA Lyon, Université Jean Monnet, ICJ UMR5208, 69622 Villeurbanne, France.
E-mail:~\texttt{antoine.detaille@univ-lyon1.fr}}}
\date{\today}

\maketitle

\begin{abstract}
	We consider the strong density problem in the Sobolev space \( W^{s,p}(Q^{m};\mathscr{N}) \) of maps with values into a compact Riemannian manifold \( \mathscr{N} \).
	It is known, from the seminal work of Bethuel, that such maps may always be strongly approximated by \( \mathscr{N} \)-valued maps that are smooth outside of a finite union of \( (m -\lfloor sp \rfloor - 1) \)-planes.
	Our main result establishes the strong density in \( W^{s,p}(Q^{m};\mathscr{N}) \) of an improved version of the class introduced by Bethuel, where the maps have a singular set \emph{without crossings}.
	This answers a question raised by Brezis and Mironescu.
	
	In the special case where \( \mathscr{N} \) has a sufficiently simple topology and for some values of \( s \) and \( p \), this result was known to follow from the \emph{method of projection}, which takes its roots in the work of Federer and Fleming.
	As a first result, we implement this method in the full range of \( s \) and \( p \) in which it was expected to be applicable.
	In the case of a general target manifold, we devise a topological argument that allows to remove the self-intersections in the singular set of the maps obtained via Bethuel’s technique.
\end{abstract}

\tableofcontents

%%%%%%%%%%%%%%%%%%%%%%%%%%%%%%%%%%%%%%%%%%%%%%%%%%%%%%%%%%%%%%%%%%%%%%%%%%%%%%%%%%%%%%%%
%%%%%%%%%%%%%%%%%%%%%%%%%%%%%%%%%%%%%%%%%%%%%%%%%%%%%%%%%%%%%%%%%%%%%%%%%%%%%%%%%%%%%%%%
%%%%%%%%%%%%%%%%%%%%%%%%%%%%%%%%%%%%%%%%%%%%%%%%%%%%%%%%%%%%%%%%%%%%%%%%%%%%%%%%%%%%%%%%
%%%%%%%%%%%%%%%%%%%%%%%%%%%%%%%%%%%%%%%%%%%%%%%%%%%%%%%%%%%%%%%%%%%%%%%%%%%%%%%%%%%%%%%%
%%%%%%%%%%%%%%%%%%%%%%%%%%%%%%%%%%%%%%%%%%%%%%%%%%%%%%%%%%%%%%%%%%%%%%%%%%%%%%%%%%%%%%%%

\section{Introduction}
\label{sect:intro}

One of the most important problems concerning the Sobolev space \( W^{s,p}\parens{\Omega;\mN} \) of maps \emph{with values into a compact manifold \( \mN \)} is the \emph{strong density problem}.
Here, \( \Omega \subset \bR^{m} \) is a sufficiently smooth bounded open set, and \( 1 \leq p < +\infty \), \( 0 < s < +\infty \).
Moreover, we split \( s = k + \sigma \), where \( k = \floor{s} \in \bN \) is the integer part of \( s \) and \( \sigma \in \lbrack0,1\rparen \) is the fractional part of \( s \).
It is known from the work of Schoen and Uhlenbeck~\cite[Section~4]{SchoenUhlenbeck1983} that, in a striking contrast with the case of real-valued maps, the set \( C^{\infty}\parens{\overline{\Omega};\mN} \) of smooth maps with values into \( \mN \) \emph{need not} be dense in \( W^{s,p}\parens{\Omega;\mN} \).
The failure of strong density of smooth maps comes from \emph{topological obstructions} due to the target manifold.
Aside from the problem of characterizing those target manifolds \( \mN \) such that \( C^{\infty}\parens{\overline{\Omega};\mN} \) is dense in \( W^{s,p}\parens{\Omega;\mN} \), depending on \( s \), \( p \), and \( \Omega \), a natural question that arises in view of this phenomenon is to find a suitable class of smooth maps outside of a small singular set that would be dense in \( W^{s,p}\parens{\Omega;\mN} \) regardless of the target \( \mN \).

A major breakthrough in this regard was accomplished by Bethuel in his seminal paper~\cite{Bethuel1991} in the case of \( W^{1,p} \), and was subsequently pursued in \( W^{s,p} \) with \( 0 < s < 1 \) by Brezis and Mironescu~\cite{BrezisMironescu2015}, in \( W^{k,p} \) with \( k = 2 \), \( 3 \dotsc \) by Bousquet, Ponce, and Van Schaftingen~\cite{BousquetPonceVanSchaftingen2015}, and in \( W^{s,p} \) with \( s > 1 \) non-integer by the author~\cite{Detaille2023}.

In order to state precisely Bethuel's theorem and its counterpart for arbitrary \( s \), we need to introduce the relevant class of functions.
But first, let us recall the precise definition of the Sobolev space \( W^{s,p}\parens{\Omega;\mN} \).
In the sequel, \( \mN \) denotes a smooth compact connected Riemannian manifold without boundary, isometrically embedded in \( \bR^{\nu} \).
The latter assumption is not restrictive, since one may always find such an embedding provided that one chooses \( \nu \in \bN \) sufficiently large; see~\cite{Nash1954,Nash1956}.
The space \( W^{s,p}(\Omega;\mN) \) is then defined as the set of all maps \( u \in W^{s,p}(\Omega;\bR^{\nu}) \) such that \( u(x) \in \mN \) for almost every \( x \in \Omega \).
Due to the presence of the manifold constraint, \( W^{s,p}(\Omega;\mN) \) is in general not a vector space, but it is nevertheless a metric space endowed with the distance defined by 
\[
	d_{W^{s,p}(\Omega)}(u,v) = \norm{u-v}_{W^{s,p}(\Omega)}.
\]

\begin{definition}
\label{def:class_Rclas}
	The class \( \Rclas_{i}\parens{\Omega;\mN} \) is the set of all maps \( u \) such that there exists a set \( \mS = \mS_{u} \subset \bR^{m} \) which is a finite union of closedly embedded \( i \)\=/dimensional submanifolds of \( \bR^{m} \) and such that \( u \in C^{\infty}\parens{\overline{\Omega} \setminus \mS;\mN} \) and 
	\[
	\abs{D^{j}u\parens{x}}
	\leq
	C\frac{1}{\dist{\parens{x,\mS}}^{j}}
	\quad
	\text{for every \( x \in \Omega \) and \( j \in \bN_{\ast} \),}
	\]
	where \( C > 0 \) is a constant depending on \( u \) and \( j \).
\end{definition}

We note importantly that the singular set \( \mS \) \emph{depends on the map \( u \)}.
Moreover, when we write \( u \in C^{\infty}\parens{\overline{\Omega} \setminus \mS;\mN} \), this means that there exists an open set \( U \subset \bR^{m} \) such that \( \Omega \Subset U \) and a map \( v \in C^{\infty}\parens{U \setminus \mS;\mN} \) such that \( u = v \) on \( \Omega \setminus \mS \).
In Definition~\ref{def:class_Rclas} and in the sequel, by \emph{submanifold}, we always implicitly mean \emph{submanifold without boundary}.
We shall always explicitly state when the submanifolds are allowed to have a boundary. 
By \emph{closedly embedded}, we mean that the manifold should be a closed subset of \( \bR^{m} \), which should not be confused with a \emph{closed manifold}, which is a compact manifold without boundary.
Observe also that, since \( \mS \) is a submanifold of the whole \( \bR^{m} \) instead of merely \( \Omega \), the estimate on the derivatives of \( u \) may depend on parts of \( \mS \) that lie outside of \( \Omega \) --- although we may always restrict to the part of \( \mS \) lying in a neighborhood of \( \overline{\Omega} \), enlarging the constant \( C \) if necessary.
This technical detail will be of crucial importance for us later on, when we require stability properties of the class \( \Rclas \) under composition with local diffeomorphisms for instance --- we omit the subscript when we want to speak about the class \( \Rclas \) in general without specifying the dimension of the singular set.

With these definitions at hand, Bethuel's theorem and its counterpart for arbitrary \( 0 < s < +\infty \) read as follows.

\begin{theorem}
\label{theorem:strong_density}
	Assume that \( sp < m \) and that \( \Omega \) satisfies the segment condition.
	The class \( \Rclas_{m-\floor{sp}-1}\parens{\Omega;\mN} \) is always dense in \( W^{s,p}\parens{\Omega;\mN} \).
\end{theorem}

Having at hand Theorem~\ref{theorem:strong_density}, a natural question is whether or not one may improve the class \( \Rclas \) to get an even better dense class of almost smooth maps.
For this purpose, we introduce the following subclass of the class \( \Rclas \).

\begin{definition}
\label{def:class_Rsmooth}
	The class \( \Rsmooth_{i}\parens{\Omega;\mN} \) is the set of all \( u \in \Rclas_{i}\parens{\Omega;\mN} \) such that the singular set \( \mS \) is a closedly embedded \( i \)\=/dimensional submanifold of \( \bR^{m} \).
\end{definition}

Our main result reads as follows.

\begin{theorem}
	\label{theorem:main}
	Assume that \( sp < m \).
	The class \( \Rsmooth_{m-\floor{sp}-1}\parens{Q^{m};\mN} \) is dense in \( W^{s,p}\parens{Q^{m};\mN} \).
\end{theorem}

Here, \( Q^{m} = \parens{-1,1}^{m} \) is the open unit cube in \( \bR^{m} \).
The key feature of Theorem~\ref{theorem:main} above is to assert that one may \emph{avoid the crossings} in the singular sets of the almost smooth maps that are dense in \( W^{s,p}\parens{Q^{m};\mN} \).
Indeed, since the singular set of a map in \( \Rclas_{m-\floor{sp}-1}\parens{Q^{m};\mN} \) is a \emph{union} of submanifolds, it may exhibit crossings at the points where those manifolds intersect.
In fact, the singular sets of the maps constructed in the existing proofs of Theorem~\ref{theorem:strong_density} for arbitrary target manifolds \emph{do} exhibit crossings, as they arise as dual skeletons of decompositions of the domain into cubes.
This will be explained in a more detailed way later on.
On the other hand, \( \Rsmooth \) corresponds to all the maps in \( \Rclas \) such that their singular set is \emph{uncrossed}, that is, does not have crossings.
This explains our choice of notation for both classes \( \Rclas \) and \( \Rsmooth \). 

We emphasize that, in full generality, Theorem~\ref{theorem:main} is new \emph{even in the case \( s = 1 \)}.
This answers in particular a question raised by Brezis and Mironescu; see e.g.\ the discussion in~\cite[Chapter~10]{BrezisMironescu2021}. 

Up to now, the only approach to prove strong density results that was able to provide the density of the class \( \Rsmooth \) instead of merely the class \( \Rclas \) was based on the \emph{method of projection}, famously devised by Federer and Fleming~\cite{FedererFleming1960} in their study of normal and integral currents.
Adapted by Hardt and Lin~\cite{HardtLin1987} in the context of maps to manifolds in order to tackle the extension problem, this method was subsequently used to prove strong density results notably by Bethuel and Zheng~\cite{BethuelZheng1988} for \( W^{1,p}\parens{\bB^{m};\bS^{m-1}} \) when \( m-1 \leq p < m \), Rivière~\cite{Riviere2000} for \( W^{\frac{1}{2},2}\parens{\bS^{2};\bS^{1}} \), Bourgain, Brezis, and Mironescu~\cite{BourgainBrezisMironescu2005} for \( W^{s,p}\parens{\bS^{m};\bS^{m-1}} \) when \( 0 < s < 1 \) and \( sp < m \) (see also~\cite{BourgainBrezisMironescu2004} for the case \( s = \frac{1}{2} \)), Bousquet~\cite{Bousquet2007} for \( W^{s,p}\parens{\bS^{m};\bS^{1}} \) when \( 1 \leq sp < 2 \), and Bousquet, Ponce, and Van Schaftingen~\cite{BousquetPonceVanSchaftingen2014} for an \( \parens{\floor{sp}-1} \)\=/connected target \( \mN \) when \( 0 < s < 1 \).
Therefore, up to now, density results for the method of projection are limited either to specific targets \( \mN \) or to the range \( 0 < s \leq 1 \).
For other closely related directions of research, see e.g.\ the work of Haj\l asz~\cite{Hajlasz1994} for a method of \emph{almost projection}, with further developments to fractional spaces by Bousquet, Ponce, and Van Schaftingen~\cite{BousquetPonceVanSchaftingen2013}, and also the work of Pakzad and Rivière~\cite{PakzadRiviere2003} concerning weak density and connections.

In the first part of our paper, in Section~\ref{sect:sing_proj}, we show that the method of projection indeed works in its full expected applicability range, that is, for any \( 0 < s < +\infty \) and any \( \parens{\floor{sp}-1} \)\=/connected target manifold \( \mN \).
This answers a question raised by Bousquet, Ponce, and Van Schaftingen; see~\cite[Section~2]{BousquetPonceVanSchaftingen2014}.
Although not allowing to prove Theorem~\ref{theorem:main} in its full generality, this result is interesting \emph{per se}: (i) it gives the full range of applicability of the method of singular projection, (ii) it provides a much simpler proof of Theorem~\ref{theorem:strong_density} in the particular case of an \( \parens{\floor{sp}-1} \)\=/connected target, and (iii) it has the advantage of applying to a general domain \( \Omega \), unlike our proof of Theorem~\ref{theorem:main}.

In order to present the additional difficulties arising when implementing the method of projection in the full range \( 0 < s < +\infty \) of fractional Sobolev spaces, let us first briefly explain how it works in the simple case of a sphere target. 
When \( \mN = \bS^{N} \subset \bR^{N+1} \), the idea of the method is to approximate an \( \bS^{N} \)\=/valued map \( u \) first by considering the convolution \( \varphi_{\eta} \ast u \), and then projecting this construction onto \( \bS^{N} \) by letting 
\[
	u_{\eta,a} = \frac{\varphi_{\eta} \ast u - a}{\abs{\varphi_{\eta} \ast u - a}}\text{.}
\]
Introducing the parameter \( a \) is at the core of Federer and Fleming's original idea for the projection method.
One should then suitably choose \( a = a_{\eta} \) such that \( a_{\eta} \to 0 \) as \( \eta \to 0 \) and establish appropriate estimates to prove that the maps \( u_{\eta} = u_{\eta,a_{\eta}} \) belong to the class \( \Rclas \) and converge to \( u \) with respect to the Sobolev norm as \( \eta \to 0 \).

The main novel difficulty in our setting is that we need to establish fractional estimates for a general singular projection.
Indeed, up to now, these estimates either were obtained by relying on the specific form of the projection for a particular target, as in~\cite{Bousquet2007}, or were deduced from the integer order estimates through the Gagliardo--Nirenberg interpolation inequality when \( 0 < s < 1 \), as in~\cite{BousquetPonceVanSchaftingen2014}.
However, for \( s > 1 \), this approach would force us to exclude some relevant values of the parameters \( s \) and \( p \).

To illustrate the need for direct estimates, let us see what can be obtained by interpolation.
Assume for instance that one wants to prove the density of the class \( \Rclas_{0}\parens{\bB^{2};\bS^{1}} \) in \( W^{s,p}\parens{\bB^{2};\bS^{1}} \) in the case \( 1 \leq sp < 2 \) --- which is the only relevant one.
One typically wants to interpolate \( W^{s,p} \) between \( L^{r} \) and \( W^{k,q} \), with \( k \in \bN \) satisfying \( k > s \).
For this to hold, one is led to choose \( r \) and \( q \) satisfying the relation
\[
%\label{eq:relation_pqr}
	\frac{1}{p} = \frac{1-\theta}{r} + \frac{\theta}{q}\text{,}
\]
where \( \theta \in \parens{0,1} \) satisfies 
\[
	s = 0\cdot\parens{1-\theta} + k\theta \text{,}
\]
that is, \( \theta = s/k \).
The key assumption to implement successfully Federer and Fleming's averaging argument over \( a \), which essentially requires that the \( W^{k,q} \)\=/norm of \( x \mapsto \frac{x}{\abs{x}} \) over \( \bB^{2} \) should be finite, is therefore that \( kq < 2 \).
If \( 0 < s < 1 \), then we may take \( k = 1 \), and hence the condition is \( q < 2 \).
But if \( 1 < s < 2 \), then \( k \geq 2 \), and this implies that \( q \) should be chosen less that \( 1 \), which is not possible.
Therefore, one sees that some ranges of values of \( s \) and \( p \) that are relevant in the problem of strong density of the class \( \Rclas \) cannot be handled by interpolation when \( s > 1 \) is not an integer.
For the record, we note that the above model case is exactly the one treated by Bousquet~\cite{Bousquet2007}, using direct fractional estimates relying on the specific form of the singular projection when the target is a circle.
Similarly, for maps \( \bB^{3} \to \bS^{2} \) with \( 2 < sp < 3 \), one cannot handle the case \( 2 < s < 3 \) by means of interpolation.
To the best of our knowledge, no direct estimates are available in the existing literature for this case, and hence the method of projection could not be implemented in this setting up to now.

Nevertheless, even knowing the full range of validity of the method of projection is not sufficient to prove Theorem~\ref{theorem:main} in full generality.
Indeed, as will be proved in Lemma~\ref{lemma:necessary_condition_projection}, a singular projection is only available when the target is \( \parens{\floor{sp}-1} \)\=/connected.
Moreover, as will be discussed in Section~\ref{subsect:non_existence_proj}, there is little hope that even all \( \parens{\floor{sp}-1} \)\=/connected manifolds admit a singular projection whose singular set is a submanifold, which is required to deduce the density of the uncrossed class \( \Rsmooth \) using the method of projection.
Therefore, proving Theorem~\ref{theorem:main} in the general case requires to find a different approach.
This is the purpose of the second part of this paper, in Section~\ref{sect:general_case}.

Before giving a sketch of our approach, we introduce some notation relative to decompositions into cubes.
Given \( r > 0 \), a \emph{cubication of radius \( r \) in \( \bR^{m} \)} is a subset of the family of cubes \( Q_{r}(a) + 2r\bZ^{m} \) for some \( a \in \bR^{m} \), 
where \( Q_{r}(a) \) denotes the cube of inradius \( r \) centered at \( a \) in \( \bR^{m} \).
Here the inradius of a cube is the half of its sidelength.
We speak about a cubication \emph{of \( \Omega \subset \bR^{m} \)} when we want to specify the set formed by the union of all cubes in the cubication.
If \( U^{m} \) is a cubication and \( \ell \in  \set{0,\dotsc,m} \), the \emph{\( \ell \)\=/skeleton} \( U^{\ell} \) of \( U^{m} \) is the set of all faces of dimension \( \ell \) of all cubes in \( U^{m} \).
An \emph{\( \ell \)\=/subskeleton  of \( U^{m} \)} is a subset of \( U^{\ell} \).
Given a skeleton \( U^{\ell} \), we write
\[
\mU^{\ell} = \bigcup_{\sigma^{\ell} \in U^{\ell}} \sigma^{\ell}
\]
the set formed by all the elements of \( U^{\ell} \).
In the sequel, we shall often make the abuse of language of also calling the underlying set \( \mU^{\ell} \) an \emph{\( \ell \)\=/skeleton}, but we shall always carefully distinguish it from the set of all \( \ell \)\=/faces using the notation that we introduced.

A special skeleton that will be of particular interest for us is the \( \ell \)\=/skeleton \( K^{\ell} \) of the unit cube \( \overline{Q^{m}} = \bracks{-1,1}^{m} \) for every \( \ell \in \set{0,\dotsc,m} \), so that \( K^{m} = \set{\overline{Q^{m}}} \). 
For every \( \ell \in \set{0,\dotsc,m-1} \), the dual skeleton of \( K^{\ell} \) is the subskeleton \( T^{\ell^{\ast}} \), where \( \ell^{\ast} = m-\ell-1 \), such that \( \mT^{\ell^{\ast}} \) is the set of all those \( x \in \overline{Q^{m}} \) that have at least \( \ell+1 \) vanishing components.
The dual skeleton of a general skeleton is then defined by taking the union of all dual skeletons of the cubes forming the skeleton.
Illustrations of skeletons (in blue) and their duals (in red) in the unit cube \( \overline{Q^{3}} \) are provided on Figure~\ref{fig:skel_and_dual}.
The value of \( \ell \) ranges from \( 2 \) on the left to \( 0 \) on the right, which corresponds to a value of \( \ell^{\ast} \) ranging from \( 0 \) to \( 2 \).

\begin{figure}[ht]
	\centering
	\includegraphics[page=1]{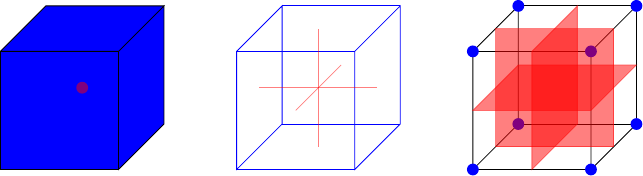}
	\caption{Skeletons and their dual skeletons}
	\label{fig:skel_and_dual}
\end{figure}

We introduce a more rigid version of the class \( \Rclas \), where the singular set is required to coincide with the dual skeleton of some cubication of \( \bR^{m} \).

\begin{definition}
	The class \( \Rrig_{i}\parens{\Omega;\mN} \) is the set of all \( u \in \Rclas_{i}\parens{\Omega;\mN} \) such that the singular set \( \mS \) is the \( i \)\=/dimensional dual skeleton of some cubication of \( \bR^{m} \).
\end{definition}

The fact that the class \( \Rrig \) --- recall that we omit the subscript when we want to speak about the class \( \Rclas \) or one of its variants without specifying the dimension of the singular set --- is a subclass of the class \( \Rclas \) is a consequence of the fact that the singular set of maps in \( \Rrig \) may be taken to be a finite union of \( i \)\=/dimensional hyperplanes.
Indeed, the dual skeleton of a cubication of \( \bR^{m} \) is a union of affine spaces, and one may keep only a finite number of them since \( \Omega \) is bounded.

It is a consequence of the proof of Theorem~\ref{theorem:strong_density} --- but not of Theorem~\ref{theorem:strong_density} itself --- that the more rigid class \( \Rrig_{m-\floor{sp}-1}\parens{\Omega;\mN} \) is dense in \( W^{s,p}\parens{\Omega;\mN} \).
Indeed, the maps in \( \Rclas_{m-\floor{sp}-1} \) that are constructed in the proof of Theorem~\ref{theorem:strong_density} to approximate a given map in \( W^{s,p} \) actually belong to \( \Rrig_{m-\floor{sp}-1} \).

To prove our main result, it therefore suffices to be able to approximate any map in \( \Rrig \).
This is the main goal of Section~\ref{sect:general_case}.
Our proof is in two main steps.
First, we devise a topological procedure that removes the crossings between the orthogonal hyperplanes constituting the singular set of a general map in \( \Rrig \).
This procedure, which itself consists of several steps, only requires to modify the initial map on a small set, but comes without any estimate.
In order to obtain, from the previous construction, a better map with suitable estimates, we rely in a second step on the \emph{shrinking} procedure from~\cite{BousquetPonceVanSchaftingen2015}, which is a more involved version of the scaling argument that was already used by Bethuel in his seminal paper~\cite{Bethuel1991} to remove the singularities with control of energy.

The new ingredient is therefore the topological procedure to uncross the singularities.
This procedure is explained in Section~\ref{subsect:part_cases} in particular cases that allow for more simple notation and illustrative figures, before the general case, presented in Section~\ref{subsect:general_procedure}.
At the core of the argument lies the following model problem.
It is well-known that the \( 1 \)\=/skeleton \( \mK^{1} \) of the unit cube \( \overline{Q^{3}} \) is a retract of \( \overline{Q^{3}} \setminus \mT^{1} \), where \( T^{1} \) is the dual skeleton of \( K^{1} \).
Is it possible to write instead \( \mK^{1} \) as a retract of \( \overline{Q^{3}} \setminus \mS\), where \( \mS \) is a \( 1 \)\=/dimensional \emph{submanifold} of \( \bR^{3} \), that is, without crossing?
Although it may come as very surprising, the answer to this question is actually \emph{yes}.
Elaborating on the construction allowing to obtain such a retraction is the cornerstone of the topological step of our proof in Section~\ref{sect:general_case}.

As a concluding remark, we comment on the dimension of the singular set in the class \( \Rclas \) and its variants.
Indeed, as we explained, the content of Theorem~\ref{theorem:main} is to provide the strong density of an improved version of the class \( \Rclas \).
Another natural idea to improve the density result given by Theorem~\ref{theorem:strong_density} would be to try reducing the dimension of the singular set, that is, to prove the density of the class \( \Rclas_{i} \) for some \( i < m-\floor{sp}-1 \).
However, it turns out that, in presence of the topological obstruction ruling out the density of \( C^{\infty}\parens{\Omega;\mN} \) in \( W^{s,p}\parens{\Omega;\mN} \), the only value of \( i \) for which \( \Rclas_{i} \) is dense in \( W^{s,p} \) is \( i = m-\floor{sp}-1 \).
For smaller \( i \), the same topological obstruction also rules out the density of the class \( \Rclas_{i} \), while for larger \( i \), \( \Rclas_{i} \) is not even a subset of \( W^{s,p} \).
See~\cite[Section~6]{BousquetPonceVanSchaftingen2015} for a detailed proof in the case where \( s \in \bN_{\ast} \).
The argument may be carried out similarly for fractional order spaces.

\subsection*{Acknowledgments}

I am deeply grateful to Petru Mironescu for drawing my attention on this problem, and for his constant advice and support.
I warmly thank Augusto Ponce for many helpful suggestions to improve the presentation of this paper.

I am thankful to Katarzyna Mazowiecka for pointing out to me the reference~\cite{Gastel2016}.

I am indebted to all the experts that took time to listen to my questions and to help me, especially when I was looking for a counterexample to Theorem~\ref{theorem:main}: Vincent Borrelli, Jules Chenal, Jacques Darné, Yves Félix, and Pascal Lambrechts.
The core of Section~\ref{subsect:non_existence_proj} has roots in a sketch of argument that Pascal Lambrechts suggested to me.
The use of Poincaré and Poincaré--Lefschetz dualities was first suggested to me by Jacques Darné.
Yves Félix had the idea of Figure~\ref{fig:projection_planar_view} in Section~\ref{sect:general_case}, and the comment below about the link with the homology of the \( 1 \)\=/skeleton of the cube arose during a discussion with Jules Chenal.

Finally, I thank the anonymous referee for very helpful comments about the first version of this text.

\section{The method of singular projection}
\label{sect:sing_proj}

\subsection{Singular projections: definitions and main result}
\label{subsect:sing_projs}

In the study of the properties of the Sobolev space \( W^{s,p}(\Omega;\mN) \), a natural approach is to work on a tubular neighborhood \( \mN + B_{\iota} \) of radius \( \iota > 0 \) sufficiently small so that the nearest point projection \( \upPi \colon \mN + B_{\iota} \to \mN \) is well-defined and smooth, and to use \( \upPi\) to bring all the constructions back to the manifold.
This approach is suitable to work with in the supercritical range \( sp \geq m \), since the continuous embedding of \( W^{s,p} \) into \( C^{0} \) --- or \( \mathrm{VMO} \) in the limiting case \( sp = m \) --- usually allows one to keep all the constructions in the tubular neighborhood \( \mN + B_{\iota} \).

To deal with the more delicate range \( sp < m \) in which one cannot guarantee that the constructions we want to perform stay in \( \mN + B_{\iota} \), a natural idea would be to look for a globally defined retraction \( P \colon \bR^{\nu} \to \mN \).
However, such an approach is hopeless, since a manifold admitting such a global retraction would be contractible as a retract of the contractible space \( \bR^{\nu} \).
But any compact manifold without boundary of dimension at least \( 1 \) has at least one non trivial homology group, and is therefore non contractible.

The idea, originally introduced by Hardt and Lin~\cite{HardtLin1987} with roots in the work of Federer and Fleming~\cite{FedererFleming1960}, is to consider instead a \emph{singular projection}, that is, a smooth map \( P \colon \bR^{\nu} \setminus \Sigma \to \mN \) satisfying \( P\restrfun{\mN} = \id_{\mN} \), where \( \Sigma \subset \bR^{\nu} \) is a suitable singular set disjoint from \( \mN \).
See the introduction for more references concerning the use of the method of projection in the context of Sobolev maps to manifolds.
In the first part of this paper, our objective is to show that the method of singular projection may indeed be implemented to solve the strong density problem in its full range of expected applicability.

We start by defining the precise notion of singular projection that we consider.

\begin{definition}
	\label{def:singular_projection}
	Let \( \ell \in \set{2,\dotsc,\nu} \).
	An \emph{\( \ell \)\=/singular projection onto \( \mN \)} is a smooth map \( P \colon \bR^{\nu} \setminus \Sigma \to \mN \)
	such that \( P_{\vert \mN} = \id_{\mN} \) and 
	\[
	\abs{D^{j}P\parens{x}} \leq C\frac{1}{\dist{\parens{x,\Sigma}}^{j}}
	\quad
	\text{for every \( x \in \bR^{\nu} \setminus \Sigma \) and \( j \in \bN_{\ast} \)}
	\]
	for some constant \( C > 0 \) depending on \( j \) and \( \mN \),
	where \( \Sigma \subset \bR^{\nu} \setminus \mN \) is either the underlying set of a finite \( \parens{\nu-\ell} \)\=/subskeleton in \( \bR^{\nu} \) or a closedly embedded \( \parens{\nu-\ell} \)-submanifold of \( \bR^{\nu} \).
\end{definition}

At this stage, the reader may wonder why we split the form of allowed singular sets for singular projections into two types, instead of considering more generally maps that are singular outside of a finite union of closedly embedded \( \parens{\nu-\ell}\)\=/submanifolds of \( \bR^{\nu} \), which would include both cases in Definition~\ref{def:singular_projection}.
The answer is given by the combination of the two following lemmas.

\begin{lemma}
	\label{lemma:necessary_condition_projection}
	If there exists a continuous map \( P \colon \bR^{\nu} \setminus \Sigma \to \mN \) such that \( P\restrfun{\mN} = \id_{\mN} \), where \( \Sigma \) is a finite union of closedly embedded \( \parens{\nu-\ell}\)\=/submanifolds of \( \bR^{\nu} \), then \( \mN \) is \( \parens{\ell-2} \)\=/connected.
\end{lemma}

\begin{lemma}
	\label{lemma:existence_of_projection}
	If \( \mN \) is \( \parens{\ell-2} \)\=/connected, then it admits an \( \ell \)\=/singular projection, whose singular set is the underlying set of a finite \( \parens{\nu-\ell} \)\=/subskeleton in \( \bR^{\nu} \).
\end{lemma}

We first comment Lemma~\ref{lemma:necessary_condition_projection} and~\ref{lemma:existence_of_projection} before giving their proof.
The combination of both these lemmas shows at the same time that the \( \parens{\ell-2} \)\=/connectedness of \( \mN \) is a \emph{necessary and sufficient} condition for the existence of an \( \ell \)\=/singular projection onto \( \mN \), and that allowing the singular set to be a finite union of closedly embedded \( \parens{\nu-\ell}\)-submanifolds of \( \bR^{\nu} \) would not have broadened the range of target manifolds admitting a singular projection.
Meanwhile, assuming that \( \Sigma \) is the underlying set of a finite \( \parens{\nu-\ell} \)\=/subskeleton in \( \bR^{\nu} \) will allow for technical simplifications in the sequel, in addition to being the natural form of singular set arising when performing the proof of Lemma~\ref{lemma:existence_of_projection}.
We also have allowed for singular sets that are given by \emph{one} submanifold of \( \bR^{\nu} \), that is, that do not exhibit crossings, because we are precisely interested in studying the density of the class \( \Rsmooth_{i} \), whose maps have a singular set without crossings.

\begin{proof}[Proof of Lemma~\ref{lemma:necessary_condition_projection}]
	The key observation is the fact that \( \bR^{\nu} \setminus \Sigma \) is \( \parens{\ell-2} \)\=/connected.
	Taking this for granted, the conclusion follows directly from the fact that a retract of an \( \parens{\ell-2} \)\=/connected space is itself \( \parens{\ell-2} \)\=/connected.
	Namely, for every \( i \in \set{0,\dotsc,\ell-2} \), we have the commutative diagram
	\[
		\begin{tikzcd}
		     & \set{0} = \pi_{i}\parens{\bR^{\nu} \setminus \Sigma} \arrow[dr, "P_{\ast}", start anchor=south east] & \\
			 \pi_{i}\parens{\mN}\arrow[ru, hookrightarrow, end anchor=south west]\arrow[rr, "\id_{\mN}"] &  & \pi_{i}\parens{\mN} 
		\end{tikzcd}
	\]
	for the maps induced between the homotopy groups.
	This implies that the identity map on the group \( \pi_{i}\parens{\mN} \) is the zero map, whence \( \pi_{i}\parens{\mN} = \set{0} \) for every \( i \in \set{0,\dotsc,\ell-2} \).
	
	The fact that \( \bR^{\nu} \setminus \Sigma \) is \( \parens{\ell-2} \)\=/connected is presumably well-known, but it seems difficult to find a proof of it in the general case, so we provide one for the convenience of the reader.
	One may consult~\cite[Lemma~3.8]{MironescuVanSchaftingen2021} for a proof in the case where \( \Sigma \) is an affine space.
	Our argument relies on the same idea.
	We show that, if \( \Omega \) is an \( \parens{\ell-2} \)\=/connected open subset of \( \bR^{\nu} \) and \( \mM \) a closedly embedded \( \parens{\nu-\ell} \)\=/submanifold of \( \Omega \), then \( \Omega \setminus \mM \) is \( \parens{\ell-2} \)\=/connected.
	The conclusion then follows by removing inductively each manifold constituting \( \Sigma \), and using this claim at each step to show that the resulting set remains \( \parens{\ell-2} \)\=/connected.
	
	To prove the claim, let \( i \in \set{0,\dotsc,\ell-2} \) and \( f \colon \bS^{i} \to \Omega \setminus \mM \) be a continuous map.
	Since \( \Omega \) is \( \parens{\ell-2} \)\=/connected, there exists a continuous map \( g \colon \overline{\bB^{i+1}} \to \Omega \) such that \( g = f \) on \( \bS^{i} \).
	Moreover, by a standard regularization process, we may assume that \( g \) is smooth on \( \bB^{i+1} \).
	Since \( f\parens{\bS^{i}} \) is closed and \( \Omega \setminus \mM \) is open, there exists \( \delta > 0 \) sufficiently small such that, for any \( a \in B_{\delta} \), \( \parens{f\parens{\bS^{i}} - a} \cap \parens{\Omega \setminus \mM} = \varnothing \).
	Now, we invoke the following particular case of Lemma~\ref{lemma:Sard} that will be proved below using Sard's theorem: since \( i+1 \leq \ell-1 < \ell \), for almost every \( a \in B_{\delta} \), we have \( g^{-1}\parens{\mM+a} = \varnothing \).
	This implies that, for any such \( a \in B_{\delta} \), \( g-a \) is a continuous extension to \( \overline{\bB^{i+1}} \) of \( f-a \), whence \( f-a \) is nullhomotopic.
	But by our choice of \( \delta \), \( f-a \) and \( f \) are homotopic, which implies that \( f \) itself is nullhomotopic.
	This concludes the proof of the lemma.
\end{proof}

\begin{proof}[Proof of Lemma~\ref{lemma:existence_of_projection}]
	We follow the approach in~\cite[Proposition~4.4]{VanSchaftingenOxfordNotes}.
	Let \( \iota > 0 \) be sufficiently small so that the nearest point projection \( \upPi \) onto \( \mN \) is well-defined on \( \mN + B_{\iota} \).
	Let \( K^{\nu} \) be a cubication of \( \bR^{\nu} \) of radius \( r > 0 \).
	Choosing \( r > 0 \) sufficiently small, we may find a subskeleton \( U^{\nu} \) of \( K^{\nu} \) such that \( \mN \subset \mU^{\nu} \subset \mN + B_{\iota/2} \).
	We let \( V^{\nu} \) be the subskeleton of \( K^{\nu} \) consisting in all cubes in \( K^{\nu} \) that do not intersect some cube in \( U^{\nu} \), and \( W^{\nu} = K^{\nu} \setminus (U^{\nu} \cup V^{\nu}) \).
	
	We define \( P \) on \( \mU^{\nu} \) by \( P = \upPi \), and on \( \mV^{\nu} \), we set \( P = b \) for some arbitrary \( b \in \mN \).
	Hence, it remains to define \( P \) on \( \mW^{\nu} \).
	We proceed by induction.
	For any \( \sigma^{0} \in W^{0} \) such that \( P \) is not yet defined on \( \sigma^{0} \), we let \( P = b \) at \( \sigma^{0} \).
	Then, for any \( \sigma^{1} \in W^{1} \) on which \( P \) is not yet defined, we use the assumption \( \pi_{0}\parens{\mN} = \set{0} \) to extend \( P \) on \( \sigma^{1} \) from its values on \( W^{0} \).
	Repeating this process up to dimension \( \ell-1 \), we define \( P \) on the whole \( \mW^{\ell-1} \).
	Finally, we use successive homogeneous extensions to extend \( P \) on \( \mW^{\nu} \setminus \mT^{\parens{\ell-1}^{\ast}} \), where \( T^{\parens{\ell-1}^{\ast}} \) is the dual skeleton to \( W^{\ell-1} \).
	Recall that the homogeneous extension to \( \overline{Q^{i}} \setminus \set{0} \) of a map \( f \) defined on \( \partial Q^{i} \) is given by \( x \mapsto f\parens{x/\abs{x}_{\infty}} \).
	Hence, a first step of homogeneous extension allows us to define \( P \) on \( \mW^{\ell} \), with one singularity at the center of each \( \ell \)\=/cell.
	A second step extends \( P \) on \( \mW^{\ell+1} \), with a singular set given by a finite union of segments, whose endpoints are located at the centers of the \( \parens{\ell+1} \)\=/cells and at the centers of the \( \ell \)\=/cells from the previous step.
	We pursue this construction up to dimension \( \nu \), each step increasing the dimension of the singular set by \( 1 \).
	By the properties of homogeneous extension, the map \( P \) that we constructed is indeed a singular projection, with singular set given by \( \Sigma = \mT^{\parens{\ell-1}^{\ast}} \), which concludes the proof.
	
	We observe however that the above argument produces only a Lipschitz map.
	To obtain a smooth map, one should slightly modify the first step, which relies on topological extension, to define smoothly \( P \) on \( \mW^{\ell-1} + B_{r/2} \) instead of merely \( \mW^{\ell-1} \).
	Then, one should use the \emph{thickening} procedure from~\cite[Section~4]{BousquetPonceVanSchaftingen2015} instead of homogeneous extension, in order to get a smooth map outside of \( \mT^{\parens{\ell-1}^{\ast}} \) with the required estimates for all derivatives.
	See also the work of Gastel~\cite[Proposition~1]{Gastel2016} for a more detailed but slightly different proof.
\end{proof}

Now that we have defined a precise notion of singular projection, we may state the main result of this section.
\begin{theorem}
	\label{thm:main_method_of_projection}
	Assume that \( \Omega \subset \bR^{m} \) is a bounded open set satisfying the segment condition, and that there exists an \( \ell \)-singular projection \( P \colon \bR^{\nu} \setminus \Sigma \to \mN \) with \( sp < \ell \).
	The class \( \Rclas_{m-\ell}(\Omega;\mN) \) is dense in \( W^{s,p}(\Omega;\mN) \).
	If in addition \( \Sigma \) is a \( (\nu-\ell) \)-submanifold of \( \bR^{\nu} \), then \( \Rsmooth_{m-\ell}(\Omega;\mN) \) is dense in \( W^{s,p}(\Omega;\mN) \).
\end{theorem}

As explained in the introduction, in the particular case where the target manifold admits an \( \ell \)\=/singular projection with \( sp < \ell \), this result provides at the same time a simpler proof of the density of the class \( \Rclas \) with crossings, which corresponds to Bethuel's theorem and its counterpart for arbitrary \( 0 <s < +\infty \), and of our main result concerning the density of the uncrossed class \( \Rsmooth \) provided that the singular set of the target manifold has no crossing.

We note the following important particular case.

\begin{corollary}
	\label{cor:density_spheres}
	Let \( sp < N+1 \).
	The class \( \Rsmooth_{m-N-1}(\Omega;\bS^{N}) \) is dense in \( W^{s,p}(\Omega;\bS^{N}) \).
\end{corollary}

The case \( N = 1 \) was already known~\cite[Theorem~2]{Bousquet2007}, but the other cases are presumably new in the general case \( 0 < s < +\infty \).
We note that, as mentioned in the introduction, the case \( s = 1 \) is already contained in~\cite{BethuelZheng1988} and the case \( 0 < s < 1 \) was proved in~\cite{BourgainBrezisMironescu2005}, see also~\cite{BousquetPonceVanSchaftingen2014}.

This corollary is also a good opportunity to emphasize that the method of singular projection \emph{does not always provide the good singular set}.
Indeed, Corollary~\ref{cor:density_spheres} above gives the same size of singular set for every \( s \) and \( p \) such that \( sp < N+1 \).
However, if \( sp < N \), since then \( \pi_{\floor{sp}}\parens{\mN} = \set{0} \), we actually have density of \emph{smooth maps} in \( W^{s,p}\parens{\Omega;\bS^{N}} \), while Corollary~\ref{cor:density_spheres} only provides the density of the class \( \Rsmooth_{m-N-1}(\Omega;\bS^{N}) \).

\begin{proof}
	We note that \( P \colon \bR^{N+1} \setminus \{0\} \to \bS^{N} \) defined by \( P(x) = \frac{x}{\lvert x \rvert} \) is an \( (N+1) \)-singular projection, and invoke Theorem~\ref{thm:main_method_of_projection}.
\end{proof}

A similar result holds for the torus \( \bT^{2} \), for which a singular \( 2 \)\=/projection whose singular set is the union of a circle inside the torus and a line passing through the hole of the torus may be constructed by hand.

\begin{corollary}
	\label{cor:density_torus}
	Let \( sp < 2 \).
	The class \( \Rsmooth_{m-2}(\Omega;\bT^{2}) \) is dense in \( W^{s,p}(\Omega;\bT^{2}) \).
\end{corollary}

Consider now the two-holed torus \( \bT^{2} \# \bT^{2} \).
Theorem~\ref{thm:main_method_of_projection} also applies to this target, but since the singular set constructed by hand --- or using Lemma~\ref{lemma:existence_of_projection} and the fact that \( \bT^{2} \# \bT^{2} \) is connected --- exhibits crossings, we only obtain the density of the class \( \Rclas_{m-2} \).
One may wonder whether or not it is possible to construct a better singular projection onto \( \bT^{2} \# \bT^{2} \) whose singular set would be a submanifold, to deduce the density of the class \( \Rsmooth_{m-2} \).
We are not able to answer this question, but the discussion in Section~\ref{subsect:non_existence_proj} suggests that there is little hope that the answer is \emph{yes}.

\subsection{Approximation by singular projection}
\label{subsect:approximation_sing_proj}

We now turn to the proof of Theorem~\ref{thm:main_method_of_projection}.
The strategy is to rely on classical approximation by convolution, and then project back the approximating maps to the target manifold using the singular projection.
Therefore, a first key step is to control the regularity of the singular set which is obtained through this process.
In addition, we need a control on the derivatives of the projected map near the singular set. 
This is the purpose of the following lemma, based on Sard's theorem and the submersion theorem.

\begin{lemma}
	\label{lemma:Sard}
	Let \( v \in C^{\infty}\parens{\Omega;\bR^{\nu}} \) and let \( \Sigma \subset \bR^{\nu} \) be a finite union of \( \parens{\nu-\ell} \)\=/dimensional submanifolds of \( \bR^{\nu} \).
	For almost every \( a \in \bR^{\nu} \), 
	\begin{enumerate}[label=(\roman*)]
		\item the set \( v^{-1}\parens{\Sigma+a} \) is a finite union of \( \parens{m-\ell} \)\=/dimensional submanifolds of \( \Omega \), one for each manifold constituting \( \Sigma \) --- or the empty set if \( \ell > m \);
		\item if \( \ell \leq m \), for every compact \( K \subset \Omega \), there exists a constant \( C > 0 \) depending on \( \Sigma \), \( v \), \( \ell \), \( K \), and \( a \) such that, for every \( x \in K \),
		\[
			\dist{\parens{x,v^{-1}\parens{\Sigma+a}}} \leq C\dist{\parens{v\parens{x},\Sigma+a}}\text{.}
		\]
	\end{enumerate}
\end{lemma}

Lemma~\ref{lemma:Sard} is a slight generalization of~\cite[Lemma~2.3]{BousquetPonceVanSchaftingen2014} to the case where \( \Sigma \) is an arbitrary union of submanifolds, not necessarily affine spaces.
The proof of the first part follows the argument given in~\cite{BousquetPonceVanSchaftingen2014}, but for the second part, we give a different proof, by contradiction.

\begin{proof}
	For (i), it suffices to consider the case where \( \Sigma \) is made of only one submanifold, as the general case then follows by taking the union over all manifolds constituting \( \Sigma \).
	Consider the map \( \upPsi \colon \Omega \times \Sigma \to \bR^{\nu} \) defined by
	\[
		\upPsi\parens{x,z} = v\parens{x} - z.
	\]
	Since \( \upPsi \) is a smooth map between smooth manifolds, Sard's theorem --- see e.g.~\cite[Chapter~II.6]{Bredon1993} --- ensures that, for almost every \( a \in \bR^{\nu} \), the linear map \( D\upPsi\parens{x,z} \colon \bR^{m} \times T_{z}\Sigma \to \bR^{\nu} \) is surjective for every \( \parens{x,z} \in \upPsi^{-1}\parens{\set{a}} \).
	If \( \ell > m \), this already implies the conclusion, since the domain of this linear map has dimension \( m+\parens{\nu-\ell} < \nu \). 
	Therefore, \( D\upPsi\parens{x,z} \) is never surjective, which forces \( \upPsi^{-1}\parens{\set{a}} = \varnothing \) for almost every \( a \in \bR^{\nu} \).
	We note that this corresponds to the \emph{easy} case of Sard's theorem, which is nothing else but the fact that the image of a smooth map is a null set when the dimension of the codomain is strictly higher than the dimension of the domain.
	
	If \( \ell \leq m \), we pursue by observing that for any \( a \in \bR^{\nu} \) such that \( D\upPsi\parens{x,z} \) is surjective,
	\[
		\bR^{\nu} = D\upPsi\parens{x,z}\bracks{\bR^{m} \times T_{z}\Sigma} = Dv\parens{x}\bracks{\bR^{m}} + T_{z}\Sigma
		\quad
		\text{for every \( \parens{x,z} \in \upPsi^{-1}\parens{\set{a}} \).}
	\]
	Furthermore, by definition, we have \( \parens{x,z} \in \upPsi^{-1}\parens{\set{a}} \) if and only if \( v(x) = z + a \in \Sigma + a \).
	Hence, we conclude that 
	\[
		\bR^{\nu} =  Dv\parens{x}\bracks{\bR^{m}} + T_{v\parens{x}}\parens{\Sigma+a}
		\quad
		\text{for every \( x \in v^{-1}\parens{\Sigma+a} \).}
	\]
	Otherwise stated, for almost every \( a \in \bR^{\nu} \), the map \( v \) is transversal to \( \Sigma + a \).
	This implies that --- see for instance~\cite[Theorem~1.39]{Warner1983} --- for almost every \( a \in \bR^{\nu} \),
	\( v^{-1}\parens{\Sigma+a} \) is a smooth submanifold of \( \bR^{m} \) of dimension \( m - \ell \).
	
	We now turn to the proof of (ii).
	Once again, it suffices to prove the assertion when \( \Sigma \) is made of one manifold, since the distance to a union of sets is the minimum of the distances to all those sets.
	We assume without loss of generality that \( a = 0 \).
	Assume by contradiction that there exists a compact set \( K \subset \Omega \) and a sequence \( \parens{x_{n}}_{n \in \bN} \) in \( K \) such that 
	\[
		\dist{\parens{x_{n},v^{-1}\parens{\Sigma}}} > n\dist{\parens{v\parens{x_{n}},\Sigma}}\text{.}
	\]
	We note that \( x_{n} \notin v^{-1}\parens{\Sigma} \), otherwise we would have \( 0 > 0 \).
	As \( K \) is compact, up to extraction, we may assume that \( x_{n} \to x \in K \) as \( n \to +\infty \).
	We observe that \( \dist{\parens{v\parens{x},\Sigma}} = 0 \), which implies that \( v\parens{x} \in \Sigma \), and hence \( x \in v^{-1}\parens{\Sigma} \).
	
	For \( n \in \bN \) sufficiently large, let \( y_{n} \) be the nearest point projection of \( x_{n} \) onto \( v^{-1}\parens{\Sigma} \).
	Since \( x_{n} \notin v^{-1}\parens{\Sigma} \), we have \( x_{n} \neq y_{n} \).
	Moreover, by construction of the nearest point projection, we know that \( x_{n} - y_{n} \in T_{y_{n}}v^{-1}\parens{\Sigma}^{\perp} \) for every \( n \in \bN \), and also \( \abs{x_{n}-y_{n}} = \dist{\parens{x_{n},v^{-1}\parens{\Sigma}}} \).
	In particular, \( y_{n} \to x \).
	Up to a further extraction, we may assume that 
	\[
		\frac{x_{n}-y_{n}}{\abs{x_{n}-y_{n}}} \to \xi \in T_{x}v^{-1}\parens{\Sigma}^{\perp}
		\quad
		\text{as \( n \to +\infty \).}
	\]
	Since \( v \) is continuously differentiable, we deduce that 
	\[
		\frac{v\parens{x_{n}} - v\parens{y_{n}}}{\dist{\parens{x_{n},v^{-1}\parens{\Sigma}}}}
		=
		\frac{v\parens{x_{n}} - v\parens{y_{n}}}{\abs{x_{n} - y_{n}}} \to Dv\parens{x}\bracks{\xi}
		\quad
		\text{as \( n \to +\infty \).}
	\]
	
	Let us note that, since we are in the situation where 
	\[
		\bR^{\nu} =  Dv\parens{x}\bracks{\bR^{m}} + T_{v\parens{x}}\Sigma\text{,}
	\] 
	we have
	\[
		\bR^{\nu} =  Dv\parens{x}\bracks{T_{x}v^{-1}\parens{\Sigma}^{\perp}} \oplus T_{v\parens{x}}\Sigma\text{.}
	\]
	Indeed, this follows from the fact that \( Dv\parens{x}\bracks{\zeta} \in T_{v\parens{x}}\Sigma \) for every \( \zeta \in T_{x}v^{-1}\parens{\Sigma} \) and a dimension argument.
	Therefore, up to replacing the usual scalar product on \( \bR^{\nu} \) with a new one, we may assume that the two subspaces involved in the above direct sum are actually orthogonal.
	This only modifies the distances by a multiplicative constant.
	Let \( \upPi_{\Sigma} \) denote the nearest point projection onto \( \Sigma \) relative to the metric induced by this new scalar product.
	
	By the triangle inequality, we write
	\[
		\frac{\abs{v\parens{x_{n}} - v\parens{y_{n}}}}{\dist{\parens{x_{n},v^{-1}\parens{\Sigma}}}}
		\leq
		\frac{\abs{v\parens{x_{n}} - \upPi_{\Sigma}\parens{v\parens{x_{n}}}}}{\dist{\parens{x_{n},v^{-1}\parens{\Sigma}}}} 
		+ 
		\frac{\abs{\upPi_{\Sigma}\parens{v\parens{x_{n}}} - \upPi_{\Sigma}\parens{v\parens{y_{n}}}}}{\dist{\parens{x_{n},v^{-1}\parens{\Sigma}}}}\text{,}
	\]
	where we made use of the fact that \( \upPi_{\Sigma}\parens{v\parens{y_{n}}} = v\parens{y_{n}} \) since \( v\parens{y_{n}} \in \Sigma \).
	We note that \( \upPi_{\Sigma}\parens{v\parens{x_{n}}} \) is well-defined for \( n \) sufficiently large, as \( v\parens{x_{n}} \) is then close to \( \Sigma \).
	The first term in the right-hand side converges to \( 0 \) as \( n \to +\infty \) by the assumption over \( \parens{x_{n}}_{n \in \bN} \), as \( \abs{v\parens{x_{n}} - \upPi_{\Sigma}\parens{v\parens{x_{n}}}} = \dist{\parens{v\parens{x_{n}},\Sigma}} \).
	Concerning the other term, since \( \upPi_{\Sigma} \circ v \) is continuously differentiable in a neighborhood of \( x \), we have
	\[
		\frac{\upPi_{\Sigma}\parens{v\parens{x_{n}}} - \upPi_{\Sigma}\parens{v\parens{y_{n}}}}{\dist{\parens{x_{n},v^{-1}\parens{\Sigma}}}}
		\to
		D\upPi_{\Sigma}\parens{v\parens{x}}[Dv\parens{x}[\xi]]
		\quad
		\text{as \( n \to +\infty \).}
	\]
	Since \( D\upPi_{\Sigma}\parens{v\parens{x}} \) is, by construction of the nearest point projection, the orthogonal projection onto \( T_{v\parens{x}}\Sigma \), we have \( D\upPi_{\Sigma}\parens{v\parens{x}}[Dv\parens{x}[\xi]] = 0 \) as a consequence of our choice of scalar product and the fact that \( \xi \in T_{x}v^{-1}\parens{\Sigma}^{\perp} \).
	
	Hence, we conclude that
	\[
		\frac{\abs{v\parens{x_{n}} - v\parens{y_{n}}}}{\dist{\parens{x_{n},v^{-1}\parens{\Sigma}}}}
		\to 0
		\quad
		\text{as \( n \to +\infty \).}
	\]
	This implies that \( Dv\parens{x}\bracks{\xi} = 0 \).
	But, since \( \xi \in T_{x}v^{-1}\parens{\Sigma}^{\perp} \) is a nonzero vector, this contradicts the fact that 
	\[
		\bR^{\nu} =  Dv\parens{x}\bracks{T_{x}v^{-1}\parens{\Sigma}^{\perp}} \oplus T_{v\parens{x}}\Sigma\text{,}
	\]
	and concludes the proof.
\end{proof}

The next lemma provides a mean value-type estimate for the derivatives of a singular projection.

\begin{lemma}
	\label{lemma:mean_value_theorem_upPi}
	Let \( \omega \subset \bR^{\nu} \) be a bounded set and \( P \colon \bR^{\nu} \setminus \Sigma \to \mN \) be a singular projection.
	For every \( x \), \( y \in \omega \setminus \Sigma \) such that \( \dist{\parens{x,\Sigma}} \leq \dist{\parens{y,\Sigma}} \) and for every \( j \in \bN_{\ast} \),
	\[
	\abs{D^{j}P\parens{x} - D^{j}P\parens{y}}
	\leq
	C\frac{\abs{x-y}}{\dist{\parens{x,\Sigma}}^{j+1}}
	\]
	for some constant \( C > 0 \) depending on \( \Sigma \) and the diameter of \( \omega \).
\end{lemma}
\begin{proof}
	We claim that there exists \( \delta > 0 \) depending only on \( \Sigma \) such that, whenever \( \abs{x-y} \leq \delta \) and \( \dist{\parens{x,\Sigma}} < \delta \), there exists a Lipschitz path \( \gamma \colon \bracks{0,1} \to \bR^{\nu} \setminus \Sigma \) satisfying \( \gamma(0) = x \), \( \gamma(1) = y \), \( \lvert \gamma' \rvert \lesssim \lvert x-y \rvert \), and \( \dist{\parens{\gamma\parens{t},\Sigma}} \geq \dist{\parens{x,\Sigma}} \) for every \( t \in \bracks{0,1} \).
	The conclusion of the lemma follows directly from this claim.
	Indeed, if \( \abs{x-y} \leq \delta \) and \( \dist{\parens{x,\Sigma}} < \delta \), it suffices to apply the mean value theorem along the path \( \gamma \) and to use the estimates on the derivatives of \( P \).
	If instead \( \abs{x-y} \geq \delta \), since \( \dist{\parens{x,\Sigma}} \) is bounded from above on \( \omega \), we have
	\[
	\abs{D^{j}P\parens{x} - D^{j}P\parens{y}}
	\lesssim
	\frac{1}{\dist{\parens{x,\Sigma}}^{j}}
	\lesssim \frac{\abs{x-y}}{\dist{\parens{x,\Sigma}}^{j+1}}.
	\]
	In the case where \( \abs{x-y} \leq \delta \) but \( \dist{\parens{x,\Sigma}} \geq \delta \), we only have to invoke the mean value theorem along the straight line between \( x \) and \( y \).
	
	We turn to the proof of the claim.
	We first assume that \( \Sigma \) is a closedly embedded submanifold of \( \bR^{\nu} \).
	We take \( R > 0 \) so large that \( \omega \subset B_{R} \), and by a compactness argument, we choose \( 0 < \delta < R \) sufficiently small so that there exist finitely many open sets \( U_{1} \), \dots, \( U_{j} \subset \bR^{\nu} \) such that for any \( z \in \Sigma \cap B_{2R} \), there exists \( i \in \{1,\dots,j\} \) with \( B_{2\delta}(z) \subset U_{i} \), and there exist diffeomorphisms \( \upPhi_{i} \colon U_{i} \to B^{\nu-\ell}_{r_{i}} \times B^{\ell}_{s_{i}} \) for some \( r_{i} \), \( s_{i} > 0 \), satisfying \( \upPhi_{i}\parens{\Sigma \cap U_{i}} = B^{\nu-\ell}_{r_{i}} \times \set{0} \) and such that for every \( a \in U_{i} \), \( \dist{\parens{a,\Sigma}} \) is given by the norm of the second component of \( \upPhi_{i}\parens{a} \).
	Choose \( z \in \Sigma \cap B_{2R} \) such that \( x \in B_{\delta}\parens{z} \), so that \( y \in B_{2\delta}\parens{z} \).
	Let \( i \in \set{1,\dots,j} \) with \( B_{2\delta}\parens{z} \subset U_{i} \).
	We observe that we may connect \( \upPhi\parens{z} \) and \( \upPhi\parens{y} \) in \( B^{\nu-\ell}_{r_{i}} \times B^{\ell}_{s_{i}} \) by a Lipschitz path \( \tilde{\gamma} \colon \bracks{0,1} \to B^{\nu-\ell}_{r_{i}} \times B^{\ell}_{s_{i}} \) with \( \abs{\gamma'} \lesssim \abs{\upPhi\parens{x}-\upPhi\parens{y}} \) and such that the norm of the second component of \( \tilde{\gamma} \) is always larger than \( \dist{\parens{x,\Sigma}} \).
	Conclusion follows by defining \( \gamma = \upPhi^{-1} \circ \tilde{\gamma} \).
	
	In the case where \( \Sigma \) is a subskeleton, we observe that one may obtain a suitable \( \gamma \) as a succession of line segments and arcs of circle.
\end{proof}

The proof of Theorem~\ref{thm:main_method_of_projection} relies on approximation by convolution.
It will be instrumental for us to establish estimates for the distance between the convoluted map and the original one, and also estimates on the derivatives of the convoluted map.
To state the required estimates in the fractional setting, we introduce the \emph{fractional derivative} as 
\[
D^{\sigma,p}v(x) = \biggl(\int_{\Omega}\frac{\lvert v(x)-v(y) \rvert^{p}}{\lvert x-y \rvert^{m+\sigma p}}\dd y\biggr)^{\frac{1}{p}}.
\]
Let also \( \varphi \in C^{\infty}_{\mathrm{c}}(\bB^{m}) \) be a fixed mollifier, that is, 
\[
\varphi \geq 0
\quad 
\text{on \( \bB^{m} \)}
\quad
\text{and}
\quad
\int_{\bB^{m}} \varphi = 1.
\]
Given \( \eta > 0 \), we define 
\[
\varphi_{\eta}(x) = \frac{1}{\eta^{m}}\varphi\Bigl(\frac{x}{\eta}\Bigr)
\quad
\text{for every \( x \in \bR^{m} \).}
\]
Lemma~\ref{lemma:fractional_estimates_convolution} corresponds to~\cite[Lemma~2.4]{BousquetPonceVanSchaftingen2014}.
We present the proof for the sake of completeness.

\begin{lemma}
	\label{lemma:fractional_estimates_convolution}
	Assume that \( 0 < \sigma < 1 \) and let \( v \in W^{\sigma,p}(\Omega;\bR^{\nu}) \).
	For every \( \eta > 0 \) and for every \( x \in \Omega \) such that \( \eta < \dist{(x,\partial\Omega)} \),
	\begin{enumerate}[label=(\roman*)]
		\item \( \lvert \varphi_{\eta} \ast v(x) - v(x) \rvert \leq C\eta^{\sigma}D^{\sigma,p}v(x) \);
		\item \( \lvert D(\varphi_{\eta} \ast v)(x) \rvert \leq C'\eta^{\sigma-1}D^{\sigma,p}v(x) \);
	\end{enumerate}
	for some constants \( C > 0 \) depending on \( \varphi \) and \( C' > 0 \) depending on \( D\varphi \).
\end{lemma}
\begin{proof}
	Jensen's inequality ensures that 
	\[
	\begin{split}
		\lvert \varphi_{\eta} \ast v(x) - v(x) \rvert
		&\leq
		\int_{B_{\eta}} \varphi_{\eta}(h)\lvert v(x-h) - v(x) \rvert\dd h \\
		&\leq
		\biggl(\int_{B_{\eta}} \varphi_{\eta}(h)\eta^{m+\sigma p}\frac{\lvert v(x-h) - v(x) \rvert^{p}}{\lvert h \rvert^{m+\sigma p}}\dd h\biggr)^{\frac{1}{p}}.
	\end{split}
	\]
	Since \( \varphi_{\eta} \lesssim \eta^{-m} \), we conclude that
	\[
		\lvert \varphi_{\eta} \ast v(x) - v(x) \rvert
		\leq
		C\eta^{\sigma}D^{\sigma,p}v(x).
	\]
	This proves the first part of the conclusion.
	
	For the second part, by differentiating under the integral, we find
	\[
		D(\varphi_{\eta} \ast v)(x) 
		=
		\int_{B_{\eta}} D\varphi_{\eta}(h)v(x-h)\dd h.
	\]
	As \( \int_{B_{\eta}} D\varphi_{\eta} = 0 \), we may write
	\[
		\lvert D(\varphi_{\eta} \ast v)(x) \rvert
		\leq
		\int_{B_{\eta}} \lvert D\varphi_{\eta}(h) \rvert \lvert v(x-h) - v(x) \rvert\dd h.
	\]
	Since \( \int_{B_{\eta}} \lvert D\varphi_{\eta} \rvert \lesssim \eta^{-1} \), Jensen's inequality ensures that
	\[
		\lvert D(\varphi_{\eta} \ast v)(x) \rvert
		\lesssim
		\frac{1}{\eta^{\frac{p-1}{p}}}\biggl(\int_{B_{\eta}} \lvert D\varphi_{\eta}(h)\rvert\eta^{m+\sigma p}\frac{\lvert v(x-h) - v(x) \rvert^{p}}{\lvert h \rvert^{m+\sigma p}}\dd h\biggr)^{\frac{1}{p}}.
	\]
	We conclude as above by using the fact that \( \lvert D\varphi_{\eta} \rvert \leq \eta^{-m-1} \).
\end{proof}

We are now ready to prove Theorem~\ref{thm:main_method_of_projection}.
As explained in the introduction, we follow the approach in~\cite{BousquetPonceVanSchaftingen2014}.
However, as we already mentioned, the range where \( s \geq 1 \) is not an integer is more difficult, as we cannot rely on interpolation, and we need to establish directly estimates on the Gagliardo seminorm.

\begin{proof}[Proof of Theorem~\ref{thm:main_method_of_projection}]
	Let \( u \in W^{s,p}(\Omega;\mN) \).
	By a standard dilation procedure, we may assume that \( u \in W^{s,p}(\omega;\mN) \) for some open set \( \omega \subset \bR^{m} \) such that \( \overline{\Omega} \subset \omega \).
	In particular, there exists \( \gamma > 0 \) such that \( \dist(\Omega;\partial\omega) > 2\gamma \).
	We note that this is the only point in the proof where we use the regularity of \( \Omega \), and that assuming merely the segment condition is sufficient to implement a dilation argument; see e.g.~\cite[Lemma~6.2]{Detaille2023}.
	Therefore, for any \( 0 < \eta \leq \gamma \), the map \( u_{\eta} = \varphi_{\eta} \ast u \) is well-defined and smooth on \( \Omega_{\gamma} = \Omega + B_{\gamma} \).
	After an extension procedure, using e.g.\ a cutoff function, we may assume that \( u_{\eta} \) is actually a smooth (non necessarily \( \mN \)\=/valued) map on the whole \( \bR^{m} \), that coincides with \( \varphi_{\eta} \ast u \) on \( \Omega_{\gamma} \).
	Hence, for any \( a \in \bR^{\nu} \), the map \( v_{\eta,a} = P \circ (u_{\eta}-a) \) satisfies \( v_{\eta,a} \in C^{\infty}(\bR^{m} \setminus S_{\eta,a};\mN) \), where \( S_{\eta,a} = u_{\eta}^{-1}(\Sigma+a) \).
	Recall that \( \Sigma \subset \bR^{\nu} \) is the singular set of the singular projection \( P \) onto \( \mN \).
	Moreover, in the case where \( \Sigma \) is a subskeleton, by adding extra cells if necessary, we may assume that it is a finite union of \( \parens{\nu-\ell} \)\=/hyperplanes.
	By Lemma~\ref{lemma:Sard}, we deduce that \( S_{\eta,a} \) is a finite union of closed submanifolds of \( \bR^{m} \) for almost every \( a \in \bR^{\nu} \),
	and actually a closed submanifold of \( \bR^{m} \) when \( \Sigma \) is a submanifold.
	Additionally, the required estimates on the derivatives of the maps \( v_{\eta,a} \) allowing to deduce that they belong to the class \( \Rclas_{m-\ell} \) follow from the Faà di Bruno formula as in~\eqref{eq:faa_di_bruno} below, combined with point~(ii) of Lemma~\ref{lemma:Sard} and the fact that \( u_{\eta} \) has bounded derivatives on \( \Omega \).
	We are going to show that, for every \( 0 < \eta \leq \gamma \), we may choose such an \( a_{\eta} \in \bR^{\nu} \) so that \( a_{\eta} \to 0 \) as \( \eta \to 0 \)
	and \( v_{\eta,a_{\eta}} \to u \) in \( W^{s,p}(\Omega) \), and this will conclude the proof of the theorem.
	
	For this purpose, we let
	\[
	\alpha = \frac{1}{4}\dist{\parens{\Sigma,\mN}}
	\]
	and we choose \( \psi \in C^{\infty}(\bR^{\nu}) \) such that 
	\begin{enumerate}[label=(\alph*)]
		\item \( \psi(x) = 0 \) if \( \dist{\parens{x,\Sigma}} \leq \alpha \);
		\item \( \psi(x) = 1 \) if \( \dist{\parens{x,\Sigma}} \geq 2\alpha \).
	\end{enumerate}
	We write
	\[
		v_{\eta,a} = w_{\eta,a} + y_{\eta,a}\text{,}
	\]
	where 
	\begin{gather*} 
		w_{\eta,a} = \psi\parens{u_{\eta}-a} v_{\eta,a} = \parens{\psi P}\circ \parens{u_{\eta}-a}
		\intertext{and}
		y_{\eta,a} = (1-\psi(u_{\eta}-a)) v_{\eta,a} = \parens{\parens{1-\psi}P} \circ \parens{u_{\eta}-a}\text{.}
	\end{gather*}
	Since the map \( \psi P \) is smooth with bounded derivatives and since \( u_{\eta} - a_{\eta} \to u \) in \( W^{s,p}(\Omega) \) 
	whenever \( a_{\eta} \to 0 \), using the compactness of \( \mN \) to get a uniform \( L^{\infty} \) bound, we deduce from the continuity of the composition operator --- see for instance~\cite[Chapter~15.3]{BrezisMironescu2021} --- that \( w_{\eta,a_{\eta}} \to u \) in \( W^{s,p}(\Omega) \) provided that we choose \( a_{\eta} \to 0 \).
	It therefore remains to prove that we can choose the \( a_{\eta} \) so that \( y_{\eta,a_{\eta}} \to 0 \) in \( W^{s,p}(\Omega) \) in order to conclude.
	
	For this purpose, we are going to show the average estimate
	\begin{equation}
	\label{eq:average_estimate}
		\int_{B_{\alpha}} \lVert y_{\eta,a} \rVert_{W^{s,p}(\Omega)}^{p}\dd a \to 0
		\quad
		\text{as \( \eta \to 0 \).}
	\end{equation}
	Taking \eqref{eq:average_estimate} for granted, we conclude the proof as follows.
	Markov's inequality ensures that
	\[
		\biggl\lvert\biggl\{ a \in B_{\alpha}\mathpunct{:} \lVert y_{\eta,a} \rVert_{W^{s,p}(\Omega)}^{p} \geq \biggl(\int_{B_{\alpha}} \lVert y_{\eta,a} \rVert_{W^{s,p}(\Omega)}^{p}\dd a\biggr)^{\frac{1}{2}} \biggr\}\biggr\rvert
		\leq 
		\biggl(\int_{B_{\alpha}} \lVert y_{\eta,a} \rVert_{W^{s,p}(\Omega)}^{p}\dd a\biggr)^{\frac{1}{2}}
		\to
		0.
	\]
	Hence, for every \( 0 < \eta \leq \gamma \), we may choose \( a_{\eta} \in B_{\alpha} \) such that \( a_{\eta} \to 0 \) 
	and 
	\[
		\lVert y_{\eta,a} \rVert_{W^{s,p}(\Omega)} \leq \biggl(\int_{B_{\alpha}} \lVert y_{\eta,a} \rVert_{W^{s,p}(\Omega)}^{p}\dd a\biggr)^{\frac{1}{2p}} \to 0\text{,}
	\]
	which proves the theorem.
	It therefore only remains to prove estimate~\eqref{eq:average_estimate}.
	
	We start by the case where \( \sigma = 0 \), and thus \( s = k \in \bN_{\ast} \).
	For every \( j \in \{1,\dotsc,k\} \), the Faà di Bruno formula ensures that 
	\[
		\lvert D^{j}y_{\eta,a}(x) \rvert
		\lesssim
		\sum_{i=1}^{j}\sum_{\substack{1\leq t_{1} \leq \cdots \leq t_{i} \\ t_{1} + \cdots + t_{i} = j}} \lvert D^{i}((1-\psi)P)(u_{\eta}(x)-a) \rvert \lvert D^{t_{1}}u_{\eta}(x) \rvert \cdots \lvert D^{t_{i}}u_{\eta}(x) \rvert.
	\]
	Since \( \psi \) has bounded derivatives, we obtain
	\begin{equation}
	\label{eq:faa_di_bruno}
		\lvert D^{j}y_{\eta,a}(x) \rvert
		\lesssim
		\sum_{i=1}^{j}\sum_{\substack{1\leq t_{1} \leq \cdots \leq t_{i} \\ t_{1} + \cdots + t_{i} = j}} \frac{1}{\dist(u_{\eta}(x)-a,\Sigma)^{i}} \lvert D^{t_{1}}u_{\eta}(x) \rvert \cdots \lvert D^{t_{i}}u_{\eta}(x) \rvert.
	\end{equation}
	As \( \mN \) is compact, we also know that  
	\begin{equation}
	\label{eq:zero_order_estimate}
		\text{\( \lvert y_{\eta,a}(x) \rvert \) is uniformly bounded with respect to \( x \), \(  \eta \), and \( a \).} 
	\end{equation}
	Moreover, by definition of \( \psi \), the map \( y_{\eta,a} \) is supported on \( \{\dist(u_{\eta}-a,\Sigma) \leq 2\alpha\} \).
	We observe that, using the fact that \( u \in \mN \) and the definition of \( \alpha \),
	\[
		\set{\dist(u_{\eta}-a,\Sigma) \leq 2\alpha}
		\subset
		\set{\dist(u_{\eta},\Sigma) \leq 3\alpha}
		\subset
		\set{\lvert u_{\eta} - u \rvert \geq \alpha}.
	\]
	
	Since \( ip \leq sp < \ell \), we have that 
	\[
		\int_{B_{\alpha}} \frac{1}{\dist(u_{\eta}(x)-a,\Sigma)^{ip}}\dd a
		= 
		\int_{B_{\alpha}+u_{\eta}(x)} \frac{1}{\dist(a,\Sigma)^{ip}}\dd a
		\leq
		\int_{B_{R}} \frac{1}{\dist(a,\Sigma)^{ip}}\dd a
		<
		+\infty\text{,}
	\]
	where \( R > 0 \) is chosen sufficiently large so that \( B_{\alpha} + u_{\eta}(x) \subset B_{R} \) for every \( x \in \Omega \) and \( 0 < \eta \leq \gamma \).
	Integrating~\eqref{eq:faa_di_bruno} and~\eqref{eq:zero_order_estimate}, and using Tonelli's theorem and the information on the support of \( y_{\eta,a} \), we deduce that 
	\[
	\int_{B_{\alpha}} \lVert y_{\eta,a} \rVert_{W^{s,p}(\Omega)}^{p}\dd a
	\lesssim
	\int_{\{\lvert u_{\eta}-u \rvert \geq \alpha\}} 1+\sum_{j=1}^{k}\sum_{i=1}^{j}\sum_{\substack{1\leq t_{1} \leq \cdots \leq t_{i} \\ t_{1} + \cdots + t_{i} = j}}
	\lvert D^{t_{1}}u_{\eta} \rvert^{p} \cdots \lvert D^{t_{i}}u_{\eta} \rvert^{p}.
	\]
	Since \( u_{\eta} \to u \) in \( L^{p}(\Omega) \), in particular \( u_{\eta} \to u \) in measure, and therefore \( \lvert\{\lvert u_{\eta}-u \rvert \geq \alpha\}\rvert \to 0 \) as \( \eta \to 0 \).
	Hölder's inequality ensures that, for \( t_{1} + \cdots + t_{i} = j \),
	\[
	\int_{\{\lvert u_{\eta}-u \rvert \geq \alpha\}} \lvert D^{t_{1}}u_{\eta} \rvert^{p} \cdots \lvert D^{t_{i}}u_{\eta} \rvert^{p} 
	\leq
	\biggl(\int_{\{\lvert u_{\eta}-u \rvert \geq \alpha\}} \lvert D^{t_{1}}u_{\eta} \rvert^{\frac{jp}{t_{1}}}\biggr)^{\frac{t_{1}}{j}} \cdots
	\biggl(\int_{\{\lvert u_{\eta}-u \rvert \geq \alpha\}} \lvert D^{t_{i}}u_{\eta} \rvert^{\frac{jp}{t_{i}}}\biggr)^{\frac{t_{i}}{j}}.
	\]
	But as \( u \in L^{\infty}(\Omega) \cap W^{k,p}(\Omega) \), we infer from the classical Gagliardo--Nirenberg inequality --- see~\cite{Gagliardo1959} and~\cite[Lecture~2]{Nirenberg1959} ---  that
	\[ 
	u \in W^{t_{\beta},\frac{kp}{t_{\beta}}}(\Omega) \subset W^{t_{\beta},\frac{jp}{t_{\beta}}}(\Omega)
	\quad
	\text{whenever \( 1 \leq t_{\beta} \leq k \).}
	\] 
	Invoking Lebesgue's lemma, we conclude that 
	\[
	\int_{B_{\alpha}} \lVert y_{\eta,a} \rVert_{W^{s,p}(\Omega)}^{p}\dd a \to 0
	\quad
	\text{as \( \eta \to 0 \),}
	\]
	which establishes estimate~\eqref{eq:average_estimate}.
	
	We now turn to the case \( 0 < \sigma < 1 \), and we assume that \( k \geq 1 \).
	Using the integer order case, we already have 
	\[
	\int_{B_{\alpha}} \lVert y_{\eta,a} \rVert_{W^{k,p}(\Omega)}^{p}\dd a \to 0
	\quad
	\text{as \( \eta \to 0 \),}
	\]
	so that it only remains to prove that 
	\[
	\int_{B_{\alpha}} \lvert D^{k}y_{\eta,a} \rvert_{W^{\sigma,p}(\Omega)}^{p}\dd a \to 0
	\quad
	\text{as \( \eta \to 0 \).}
	\]
	Since \( y_{\eta,a} \) is supported on \( \set{\dist(u_{\eta}-a,\Sigma) \leq 2\alpha} \subset
	\set{\lvert u_{\eta} - u \rvert \geq \alpha} \), we may write
	\begin{multline*}
	\lvert D^{k}y_{\eta,a} \rvert_{W^{\sigma,p}(\Omega)}^{p} \\
	\begin{aligned}
	&\leq
	2\int_{\set{\dist(u_{\eta}\parens{x}-a,\Sigma) \leq 2\alpha}}\int_{\{\dist(u_{\eta}(x)-a,\Sigma) \leq \dist(u_{\eta}(y)-a,\Sigma)\}}
	\frac{\lvert D^{k}y_{\eta,a}(x) - D^{k}y_{\eta,a}(y) \rvert^{p}}{\lvert x-y \rvert^{m+\sigma p}}\dd y\dd x \\
	&\leq
	2\int_{\set{\lvert u_{\eta}\parens{x} - u\parens{x} \rvert \geq \alpha}}\int_{\{\dist(u_{\eta}(x)-a,\Sigma) \leq \dist(u_{\eta}(y)-a,\Sigma)\}}
	\frac{\lvert D^{k}y_{\eta,a}(x) - D^{k}y_{\eta,a}(y) \rvert^{p}}{\lvert x-y \rvert^{m+\sigma p}}\dd y\dd x.
	\end{aligned}
	\end{multline*}
	Given \( x \), \( y \in \Omega \), using the Faà di Bruno formula and the multilinearity of the derivative, we obtain
	\begin{equation}
	\label{eq:difference_derivatives_splitted}
		\abs{D^{k}y_{\eta,a}\parens{x} - D^{k}y_{\eta,a}\parens{y}}
		\lesssim
		\sum_{j=1}^{k}\sum_{\substack{1\leq t_{1} \leq \cdots \leq t_{j} \\ t_{1} + \cdots + t_{j} = k}}\parens[\bigg]{A_{j,t_{1},\dotsc,t_{j}}
		+
		\sum_{i=1}^{j} B_{i,j,t_{1},\dotsc,t_{j}}}\text{,}
	\end{equation}
	where
	\begin{gather*}
		A_{j,t_{1},\dotsc,t_{j}}
		=
		\lvert D^{j}((1-\psi)P)(u_{\eta}(x)-a) - D^{j}((1-\psi)P)(u_{\eta}(y)-a) \rvert \lvert D^{t_{1}}u_{\eta}(x) \rvert \cdots \lvert D^{t_{j}}u_{\eta}(x) \rvert
	\intertext{and}
		B_{i,j,t_{1},\dotsc,t_{j}}
		=
		\lvert D^{j}((1-\psi)P)(u_{\eta}(y)-a) \rvert \parens[\bigg]{\prod_{\substack{1 \leq \beta < i}} \lvert D^{t_{\beta}}u_{\eta}(x) \rvert} \lvert D^{t_{i}}u_{\eta}(x) - D^{t_{i}}u_{\eta}(y) \rvert \parens[\bigg]{\prod_{i < \beta \leq j} \lvert D^{t_{\beta}}u_{\eta}(y) \rvert}.
	\end{gather*}
	To bear in mind more readable terms, the reader may think of the case \( j = 1 \), where one has
	\begin{gather*}
		A_{1} = \abs{D((1-\psi)P)(u_{\eta}(x)-a) - D((1-\psi)P)(u_{\eta}(y)-a)}\abs{Du_{\eta}\parens{x}}
	\intertext{and}
		B_{1} = \abs{D((1-\psi)P)(u_{\eta}(y)-a)}\abs{Du_{\eta}\parens{x}-D_{\eta}\parens{y}}.
	\end{gather*}
	We observe that \( \rvert D^{t}u_{\eta} \lvert \lesssim \eta^{-t} \) for every \( t \in \bN \).
	Therefore,~\eqref{eq:difference_derivatives_splitted} yields
	\begin{equation}
	\label{eq:improved_difference_derivatives_splitted}
		\abs{D^{k}y_{\eta,a}\parens{x} - D^{k}y_{\eta,a}\parens{y}}^{p}
		\lesssim
		\sum_{j=1}^{k}\parens[\bigg]{C_{j} + \sum_{t=1}^{k}D_{j,t}}\text{,}
	\end{equation}
	where
	\begin{gather*}
		C_{j}
		=
		\eta^{-kp}\abs{D^{j}((1-\psi)P)(u_{\eta}(x)-a) - D^{j}((1-\psi)P)(u_{\eta}(y)-a)}^{p}
		\intertext{and}
		D_{j,t}
		=
		\eta^{-(k-t)p}\abs{D^{j}((1-\psi)P)(u_{\eta}(y)-a)}^{p}\abs{D^{t}u_{\eta}(x)-D^{t}u_{\eta}(y)}^{p}.
	\end{gather*}
	
	As for the integer case, since \( jp \leq kp < \ell \), we have
	\[
	\int_{B_{\alpha}} \lvert D^{j}((1-\psi)P)(u_{\eta}(y)-a) \rvert^{p}\dd a
	\lesssim
	\int_{B_{\alpha}} \frac{1}{\dist(u_{\eta}(y)-a,\Sigma)^{jp}}\dd a
	< 
	+\infty.
	\]
	We now integrate with respect to \( x \) and \( y \), split the integral in \( y \) into two parts, and use again the estimate \( \rvert D^{t}u_{\eta} \lvert \lesssim \eta^{-t} \).
	This yields, for any \( r > 0 \),
	\begin{multline*}
		\int_{B_{\alpha}}\int_{\{\lvert u_{\eta}(x)-u(x) \rvert \geq \alpha\}}\int_{\Omega}
		\frac{D_{j,t}}{\lvert x-y \rvert^{m+\sigma p}}\dd y\dd x\dd a \\
		\begin{aligned}
			&
			\begin{multlined}
				\lesssim
				\eta^{-(k+1)p}\int_{\{\lvert u_{\eta}(x)-u(x) \rvert \geq \alpha\}}\int_{B_{r}(x)}\frac{1}{\lvert x-y \rvert^{m+(\sigma-1)p}}\dd y \dd x \\
				+ \eta^{-kp}\int_{\{\lvert u_{\eta}(x)-u(x) \rvert \geq \alpha\}}\int_{\bR^{m} \setminus B_{r}(x)}\frac{1}{\lvert x-y \rvert^{m+\sigma p}}\dd y \dd x 
			\end{multlined} \\
			&\lesssim
			(\eta^{-(k+1)p}r^{(1-\sigma)p}+\eta^{-kp}r^{-\sigma p})\lvert \{\lvert u_{\eta}-u \rvert \geq \alpha\} \rvert.
		\end{aligned}
	\end{multline*}
	Inserting \( r = \eta \), we obtain
	\[
		\int_{B_{\alpha}}\int_{\{\lvert u_{\eta}(x)-u(x) \rvert \geq \alpha\}}\int_{\Omega}
		\frac{D_{j,t}}{\lvert x-y \rvert^{m+\sigma p}}\dd y\dd x\dd a
		\lesssim
		\eta^{-sp}\lvert \{\lvert u_{\eta}-u \rvert \geq \alpha\} \rvert.
	\]
	By the fractional Gagliardo--Nirenberg inequality --- see e.g.~\cite[Corollary~3.2]{BrezisMironescu2001} and~\cite[Theorem~1]{BrezisMironescu2018} --- we have \( u \in W^{\sigma,\frac{sp}{\sigma}}(\Omega) \).
	Invoking the Markov inequality and Lemma~\ref{lemma:fractional_estimates_convolution}, we find
	\begin{equation}
	\label{eq:estimate_tchebyshev}
	\lvert \{\lvert u_{\eta}-u \rvert \geq \alpha\} \rvert
	\leq
	\frac{1}{\alpha^{\frac{sp}{\sigma}}}\int_{\{\lvert u_{\eta}-u \rvert \geq \alpha\}} \lvert u_{\eta} - u \rvert^{\frac{sp}{\sigma}}
	\lesssim
	\eta^{sp}\int_{\{\lvert u_{\eta}-u \rvert \geq \alpha\}} (D^{\sigma,\frac{sp}{\sigma}}u)^{\frac{sp}{\sigma}}.
	\end{equation}
	Hence, using Lebesgue's lemma, we conclude that 
	\[
	\eta^{-sp}\lvert \{\lvert u_{\eta}-u \rvert \geq \alpha\} \rvert
	\lesssim
	\int_{\{\lvert u_{\eta}-u \rvert \geq \alpha\}} (D^{\sigma,\frac{sp}{\sigma}}u)^{\frac{sp}{\sigma}}
	\to
	0
	\quad
	\text{as \( \eta \to 0 \).}
	\]
	This achieves to estimate the second term in~\eqref{eq:improved_difference_derivatives_splitted}.
	
	For the first term, we also split the integral with respect to \( y \) into two parts, but this time we use Lemma~\ref{lemma:mean_value_theorem_upPi} to estimate \( \abs{D^{j}((1-\psi)P)(u_{\eta}(x)-a) - D^{j}((1-\psi)P)(u_{\eta}(y)-a)} \) in the ball \( B_{r}\parens{x} \).
	This yields, for every \( r > 0 \),
	\begin{multline*}
	\int_{B_{\alpha}}\int_{\set{\lvert u_{\eta}\parens{x} - u\parens{x} \rvert \geq \alpha}}\int_{\{\dist(u_{\eta}(x)-a,\Sigma) \leq \dist(u_{\eta}(y)-a,\Sigma)\}} 
	\frac{C_{j}}{\lvert x-y \rvert^{m+\sigma p}}\dd y\dd x\dd a \\
	\begin{multlined}
	\lesssim
	\eta^{-kp}\int_{\{\lvert u_{\eta}(x)-u(x) \rvert \geq \alpha\}} \int_{B_{\alpha}}
	\biggl(\int_{B_{r}(x)}\frac{1}{\dist(u_{\eta}(x)-a,\Sigma)^{(j+1)p}}\frac{\lvert u_{\eta}(x)-u_{\eta}(y) \rvert^{p}}{\lvert x-y \rvert^{m+\sigma p}}\dd y \\ 
	+ \int_{\bR^{m}\setminus B_{r}(x)}\frac{1}{\dist(u_{\eta}(x)-a,\Sigma)^{jp}}\frac{1}{\lvert x-y \rvert^{m+\sigma p}}\dd y\biggr)\dd a\dd x.
	\end{multlined}
	\end{multline*}
	We estimate
	\begin{gather*}
	\int_{B_{r}(x)}\frac{1}{\dist(u_{\eta}(x)-a,\Sigma)^{(j+1)p}}\frac{\lvert u_{\eta}(x)-u_{\eta}(y) \rvert^{p}}{\lvert x-y \rvert^{m+\sigma p}}\dd y
	\lesssim
	r^{(1-\sigma)p}\eta^{-p}\frac{1}{\dist(u_{\eta}(x)-a,\Sigma)^{(j+1)p}}
	\intertext{and}
	\int_{\bR^{m}\setminus B_{r}(x)}\frac{1}{\dist(u_{\eta}(x)-a,\Sigma)^{jp}}\frac{1}{\lvert x-y \rvert^{m+\sigma p}}\dd y
	\lesssim
	r^{-\sigma p}\frac{1}{\dist(u_{\eta}(x)-a,\Sigma)^{jp}}.
	\end{gather*}
	Inserting \( r = \eta\dist(u_{\eta}(x)-a,\Sigma) \), we obtain
	\begin{multline*}
		\int_{B_{\alpha}}\int_{\set{\lvert u_{\eta}\parens{x} - u\parens{x} \rvert \geq \alpha}}\int_{\{\dist(u_{\eta}(x)-a,\Sigma) \leq \dist(u_{\eta}(y)-a,\Sigma)\}} 
		\frac{C_{j}}{\lvert x-y \rvert^{m+\sigma p}}\dd y\dd x\dd a \\
		\lesssim
		\eta^{-sp}\int_{\{\lvert u_{\eta}(x)-u(x) \rvert \geq \alpha\}} \int_{B_{\alpha}}
		\frac{1}{\dist(u_{\eta}(x)-a,\Sigma)^{(j+\sigma)p}}\dd a\dd x 
		\lesssim
		\eta^{-sp}\lvert \{\lvert u_{\eta}-u \rvert \geq \alpha\} \rvert\text{,}
	\end{multline*}
	where, in the last inequality, we once more made use of the fact that \( sp < \ell \) .
	We observe interestingly that our choice of \( r \) is not so common.
	Indeed, in such an optimization argument, one usually takes \( r \) to be some suitable power of \( \eta \).
	Here, not only our choice is more complex, but it also depends on \( x \) and \( a \), the outer variables of integration.
	Using estimate~\eqref{eq:estimate_tchebyshev}, we conclude that the above quantity goes to \( 0 \) as \( \eta \to 0 \), which finishes to estimate the first term in~\eqref{eq:improved_difference_derivatives_splitted}.
	Both terms being controlled, this establishes average estimate~\eqref{eq:average_estimate}, therefore concluding the proof of the theorem when \( k \geq 1 \) and \( 0 < \sigma < 1 \).
	
	The case \( k = 0 \) and \( 0 < \sigma < 1 \) is similar, and actually simpler.
	Indeed, as no derivatives are involved, we have to estimate the difference \( ((1-\psi)P)(u_{\eta}(x)-a) - ((1-\psi)P)(u_{\eta}(y)-a) \), which is directly performed with the same technique as for the \( C_{j} \) term in the previous case.
	Moreover, this range of parameters was already treated in~\cite{BousquetPonceVanSchaftingen2014} with a different technique, interpolating with the first order term using the Gagliardo--Nirenberg inequality.
	We therefore omit the details of the argument.
\end{proof}

\subsection{Concluding thoughts: What singular set can we hope for?}
\label{subsect:non_existence_proj}

We conclude this section by considering the question of existence of a singular projection whose singular set is a closed submanifold of \( \bR^{\nu} \).
We have seen in Lemmas~\ref{lemma:necessary_condition_projection} and~\ref{lemma:existence_of_projection} that the \( \parens{\ell-2} \)\=/connectedness of the target manifold \( \mN \) is a necessary and sufficient condition for the existence of a singular projection, and that the proof produces a singular projection whose singular set is a subskeleton, and therefore exhibits crossings.
Since projections whose singular set do not have crossings allow to obtain the density of the class \( \Rsmooth \) instead of the class \( \Rclas \), it is natural to ask whether or not it is always possible to improve singular projections so that their singular set is a submanifold.
That is: Does every \( \parens{\ell-2} \)\=/connected manifold admit a singular projection whose singular set is a submanifold?

Although we are not able to answer this question, we give in this section a family of examples suggesting that there is little hope that the answer is \emph{yes}.
For every \( \ell \in \bN_{\ast} \), we let \( \mN_{\ell} \) denote a connected sum of \( \ell \) copies of the \( 2 \)\-/dimensional torus, embedded into \( \bR^{3} \).
Since \( \mN_{\ell} \) is connected, it admits a \( 2 \)\-/ singular projection.
Actually, this projection may even be taken to be the nearest point projection.
For \( \mN_{1} = \bT^{2} \), its singular set is the circle forming the core of the torus and a line passing through the hole of the torus.
For \( \mN_{2} = \bT^{2} \# \bT^{2} \), the two-holed torus, its singular set is the eight-figure forming the core of the torus and two lines, each one passing through one of the holes of \( \mN_{2} \).
One may notice that, in those two examples, the singular set of the natural singular projection onto \( \bT^{2} \) is a \( 1 \)\-/dimensional submanifold of \( \bR^{3} \), while the singular set of the natural singular projection onto \( \bT^{2} \# \bT^{2} \) is only a finite union of \( 1 \)\-/dimensional submanifolds of \( \bR^{3} \).
It is therefore natural to ask whether or not this can be improved to have a singular projection onto \( \bT^{2} \# \bT^{2} \) whose singular set would be a \( 1 \)\-/dimensional submanifold of \( \bR^{3} \).
The same question arises for \( \mN_{\ell} \) for every \( \ell \geq 2 \).

\begin{proposition}
\label{prop:no_homotopy_retract_torus}
	If \( \ell \geq 2 \), then there is no homotopy retract \( P \colon \bS^{3} \setminus \mS \to \mN_{\ell} \) such that \( \mS \) is a \( 1 \)\=/dimensional submanifold of \( \bS^{3} \).
\end{proposition}

We have stated Proposition~\ref{prop:no_homotopy_retract_torus} with \( \bS^{3} \) instead of \( \bR^{3} \), but this is equivalent up to compactification in the case of maps that are constant at infinity --- or if the singular set passes through the point at infinity, as it is the case for the \( \mN_{\ell} \) above.
In Definition~\ref{def:singular_projection}, singular projections were required to be continuous retracts of \( \bR^{\nu} \setminus \Sigma \) into \( \mN \), that is, \( P \circ i \colon \mN \to \mN = \id_{\mN} \), where \( i \) is the inclusion of \( \mN \) into \( \bR^{\nu} \setminus \Sigma \).
In Proposition~\ref{prop:no_homotopy_retract_torus}, we consider instead homotopy retracts, that is, \( P \) should in addition satisfy that \( i \circ P \colon \bR^{\nu} \setminus \Sigma \to \bR^{\nu} \setminus \Sigma \) is homotopic to the identity map.
In particular, \( P \) induces a homotopy equivalence between \( \bR^{\nu} \setminus \Sigma \) and \( \mN \).
This is a stronger requirement than asking merely that \( P \) is a continuous retract.
Nevertheless, one may check that the usual constructions for a singular projection, like in Lemma~\ref{lemma:existence_of_projection}, produce a homotopy retract that is constant at infinity, so Proposition~\ref{prop:no_homotopy_retract_torus} leaves little hope to find a singular projection into \( \mN_{\ell} \) whose singular set would be a submanifold when \( \ell \geq 2 \).

\begin{proof}
	Assume by contradiction that there exists a homotopy retract \( P \colon \bS^{3} \setminus \mS \to \mN_{\ell} \), where \( \mS \) is a submanifold of \( \bS^{3} \).
	We start by computing the homology groups of \( \bS^{3} \), \( \bS^{3} \setminus \mS \), and the relative homology groups of \( \bS^{3} \) relatively to \( \bS^{3} \setminus \mS \).
	The first homology groups of the sphere are given by
	\[
		H_{0}\parens{\bS^{3}} = \bZ,
		\quad
		H_{1}\parens{\bS^{3}} = \set{0},
		\quad
		H_{2}\parens{\bS^{3}} = \set{0},
		\quad
		H_{3}\parens{\bS^{3}} = \bZ.
	\]
	We note that we always implicitly consider homology with integer coefficients.
	On the other hand, since we assumed the existence of the homotopy retract \( P \), it follows that \( \bS^{3} \setminus \mS \) and \( \mN_{\ell} \) share the same homology groups: \( H_{j}\parens{\bS^{3} \setminus \mS} = H_{j}\parens{\mN_{\ell}} \) for every \( j \in \bN \).
	Therefore, we obtain
	\[
		H_{0}\parens{\bS^{3} \setminus \mS} = \bZ,
		\quad
		H_{1}\parens{\bS^{3} \setminus \mS} = \bZ^{2\ell},
		\quad
		H_{2}\parens{\bS^{3} \setminus \mS} = \bZ,
		\quad
		H_{3}\parens{\bS^{3} \setminus \mS} = \set{0}.
	\]
	To obtain the homology groups \( H_{k}\parens{\bS^{3},\bS^{3} \setminus \mS} \), we use the long exact sequence of relative homology groups
	\[
		\begin{tikzcd}[cramped]
			\cdots \arrow[r] & H_{k}\parens{\bS^{3}} \arrow[r] & H_{k}\parens{\bS^{3},\bS^{3} \setminus \mS} \arrow[r] & H_{k-1}\parens{\bS^{3} \setminus \mS} \arrow[r] & H_{k-1}\parens{\bS^{3}} \arrow[r] & \cdots\text{.}
		\end{tikzcd}
	\]
	The portion of this sequence for \( k = 2 \) yields
	\[
		\begin{tikzcd}
		\set{0} \arrow[r] & H_{2}\parens{\bS^{3},\bS^{3} \setminus \mS} \arrow[r] & \bZ^{2\ell} \arrow[r] & \set{0}\text{,}
		\end{tikzcd}
	\]
	which implies that \( H_{2}\parens{\bS^{3},\bS^{3} \setminus \mS} = \bZ^{2\ell} \).
	We now examine the portion of the sequence with \( k = 3 \), which translates into
	\[
		\begin{tikzcd}
		\set{0} \arrow[r] & \bZ \arrow[r] & H_{3}\parens{\bS^{3},\bS^{3} \setminus \mS} \arrow[r] & \bZ \arrow[r] & \set{0}\text{.}
		\end{tikzcd}
	\]
	As \( \bZ \) is a free \( \bZ \)\=/module, the above short exact sequence of abelian groups splits, which implies that necessarily \( H_{3}\parens{\bS^{3},\bS^{3} \setminus \mS} = \bZ \oplus \bZ \).
	
	We now recall two important duality principles concerning homology groups.
	The first one is \emph{Poincaré duality}: If \( \mM \) is a closed orientable \( m \)\-/dimensional manifold, then the homology group \( H_{m-k}\parens{\mM} \) is isomorphic to the cohomology group \( H^{k}\parens{\mM} \) for every \( k \in \set{0,\dotsc,m} \); see e.g.~\cite[Theorem~3.30]{Hatcher2002}.
	The second one is the \emph{Poincaré--Lefschetz duality}: If \( K \) is a compact locally contractible subspace of a closed orientable \( m \)\=/dimensional manifold \( \mM \), then \( H_{k}\parens{\mM, \mM \setminus K} \cong H^{m-k}\parens{K} \) for every \( k \in \set{0,\dotsc,m} \); see e.g.~\cite[Theorem~3.44]{Hatcher2002}.
	Applied to \( \mM = \bS^{3} \) and \( K = \mS \), the Poincaré--Lefschetz duality yields 
	\[
		H^{0}\parens{\mS} \cong H_{3}\parens{\bS^{3},\bS^{3} \setminus \mS} = \bZ \oplus \bZ
		\quad
		\text{and}
		\quad
		H^{1}\parens{\mS} \cong H_{2}\parens{\bS^{3},\bS^{3} \setminus \mS} = \bZ^{2\ell}\text{.}
	\]
	On the other hand, since \( \mS \) is assumed to be a \( 1 \)\=/dimensional submanifold of \( \bS^{3} \), the Poincaré duality implies that 
	\[
		H_{0}\parens{\mS} \cong H^{1}\parens{\mS} \cong \bZ^{2\ell}
		\quad
		\text{and}
		\quad
		H_{1}\parens{\mS} \cong H^{0}\parens{\mS} \cong \bZ \oplus \bZ.
	\]
	However, for a \( 1 \)\=/dimensional manifold, the groups \( H_{0} \) and \( H_{1} \) both coincide with a direct sum of the same number of copies of \( \bZ \), one for each connected component.
	Therefore, the above situation is only possible for \( \ell = 1 \), which concludes the proof.
\end{proof}

We note along the way that, when \( \ell = 1 \), the above proof shows that the singular set of a homotopy retract to \( \bT^{2} \) must have exactly two connected components.
This is coherent with what we obtain with the natural construction described above, and shows that the singular set obtained there cannot be improved to be made of only one connected component.

A similar reasoning could be carried out in other situations, provided one is able to compute the required homology groups.
For instance, one could examine the situation for non orientable surfaces, relying on homology with coefficients in \( {\bZ} / {2\bZ} \) so that Poincaré duality is also available.

\section{The general case: the crossings removal procedure}
\label{sect:general_case}

\subsection{The idea of the method}
\label{subsect:part_cases}

In this section, we consider the case of a general target manifold \( \mN \), non necessarily \( \parens{\floor{sp}-1} \)\=/connected.
In this context where the method of projection cannot be applied, all currently available proofs of the density of the class \( \Rclas \) and its variants rely on modifying the map \( u \in W^{s,p}\parens{\Omega;\mN} \) to be approximated on its domain --- in contrast with the method of projection, which consists in working on the codomain.
In the most general case, there are essentially two ideas of proof.
The first one is the method of good and bad cubes, introduced by Bethuel~\cite{Bethuel1991} to handle the case \( W^{1,p} \), and later pursued in the general case \( W^{s,p} \) with \( 0 < s < +\infty \)~\cite{BousquetPonceVanSchaftingen2015, Detaille2023}.
The second one is the averaging argument devised by Brezis and Mironescu~\cite{BrezisMironescu2015}, suited for the case \( 0 < s < 1 \).

Both these ideas require to decompose the domain \( Q^{m} \) into a small grid, and rely crucially on homogeneous extension.
In Bethuel's approach, this procedure is used to approximate \( u \) on the bad cubes of the grid, while in Brezis and Mironescu's approach, it is used on all the cubes of the grid.
By the very definition of homogeneous extension, it is clear that this technique \emph{necessarily} produces maps whose singular set exhibits crossings --- except in the case of point singularities.

The key ingredient in the homogeneous extension procedure is the standard retraction \( \overline{Q^{m}} \setminus \mT^{\ell^{\ast}} \to \mK^{\ell} \), where we recall that \( K^{\ell} \) is the \( \ell \)\=/skeleton of the unit cube, and \( T^{\ell^{\ast}} \) its dual skeleton.
In order to perform approximation with maps having a singular set \emph{without} crossings, a natural question would be whether or not there exists another retraction \( g \colon \overline{Q^{m}} \setminus \mS \to \mK^{\ell} \), where here \( \mS \) would be an \( \ell^{\ast} \)\=/submanifold of \( \bR^{m} \), that is, without crossings.
This would correspond to a modified version of the usual retraction \( \overline{Q^{m}} \setminus \mT^{\ell^{\ast}} \to \mK^{\ell} \), where the singular set has been uncrossed.

It turns out that such a retraction \emph{does} exist, and is actually quite simple to construct.
This may come as very surprising, in view of Proposition~\ref{prop:no_homotopy_retract_torus}.
We note importantly that this is not due to the fact that Proposition~\ref{prop:no_homotopy_retract_torus} requires a homotopy retract, since the map that we are going to construct is actually a homotopy retract.
The possibility to obtain such a retraction is instead due to the fact that here, we only require it to be a retraction on a \( 1 \)\=/dimensional set, while in Proposition~\ref{prop:no_homotopy_retract_torus}, we imposed a \( 2 \)\=/dimensional constraint.
This allows for more freedom in our construction.

The procedure to build this retraction \( g \) is explained below, with \( m = 3 \) to allow for illustration.
The starting point is the zero-homogeneous map \( x \mapsto x/\abs{x}_{\infty} \), which retracts \( \overline{Q^{3}} \setminus \set{0} \) onto \( \partial Q^{3} \).
Choosing the center of projection to be a point \emph{above} \( Q^{3} \) instead of \emph{inside} \( Q^{3} \) yields a continuous map \( h \) defined on the whole \( \overline{Q^{3}} \), that retracts \( \overline{Q^{3}} \) onto its four lateral faces and its lower face.
We then postcompose the map \( h \) with the usual retraction of these five faces minus their centers onto their boundary, which is exactly \( \mK^{1} \).
This produces the expected continuous retraction \( g \colon \overline{Q^{3}} \setminus \mS \to \mK^{1} \), where \( \mS \) is the inverse image of the centers of the five aforementioned faces under \( h \), which consists of five line segments that emanate from those centers and end up on the top face of \( Q^{3} \).
Those lines do not cross inside \( Q^{3} \), but they would intersect at the center of projection above \( Q^{3} \) if they were extended up to there. 
The situation is depicted on Figure~\ref{fig:retraction_g}, where the singularities of \( g \) are represented in red, and extended up to the projection point to help visualization. 

\begin{figure}[ht]
	\centering
	\includegraphics[page=2]{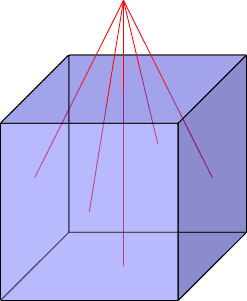}
	\caption{The retraction \( g \) and its singular set}
	\label{fig:retraction_g}
\end{figure}

Another way of looking at this construction is the following.
Viewed from the projection point lying slightly above \( Q^{3} \), the set of all faces except the top one looks like on Figure~\ref{fig:projection_planar_view}, with the centers of the faces represented in red.
The retraction \( g \colon \overline{Q^{3}} \setminus \mS \to \mK^{1} \) may then be viewed as a vertical projection onto the set depicted on Figure~\ref{fig:projection_planar_view}, followed by the retraction onto the edges away from the red centers.
The singular set would then look like vertical lines starting from the red centers.

\begin{figure}[ht]
	\centering
	\includegraphics[page=3]{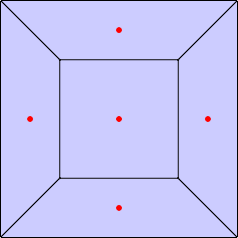}
	\caption{Planar view of the faces of \( Q^{3} \) without the top one}
	\label{fig:projection_planar_view}
\end{figure}

As a final comment concerning this model construction, we note that it appears to be natural in the context of homology theory.
Indeed, the first homology group of \( \mK^{1} \) is given by \( H_{1}\parens{\mK^{1}} = \bZ^{5} \), with one cycle generated by the boundary of each face of \( \mK^{1} \) except the top one which is the sum of all five others.
This is clearly seen on \( \overline{Q^{3}} \setminus \mS \): there is one cycle winding around each segment constituting \( \mS \), each one corresponding to the boundary of one of the five lowest faces of \( \overline{Q^{3}} \), and the sum of all of them corresponds to the boundary of the top face.
This suggests that our construction is somehow adapted to the homology of \( \mK^{1} \).
Moreover, this can be used to prove that the set \( \mS \) must have exactly five connected components \emph{inside of \( Q^{3} \)}, so that our construction is optimal in this sense.

Having at our disposal the smooth retraction \( g \) is a first step towards the proof of the density of the class of maps with uncrossed singular set \( \Rsmooth_{m-\floor{sp}-1} \) in \( W^{s,p} \), but we are not done yet.
As the more rigid class \( \Rrig_{m-\floor{sp}-1}\parens{Q^{m};\mN} \) is dense in \( W^{s,p}\parens{Q^{m};\mN} \), it suffices to show that every map that belongs to \( \Rrig_{m-\floor{sp}-1}\parens{Q^{m};\mN} \) may be approximated in \( W^{s,p} \) by maps in \( \Rsmooth_{m-\floor{sp}-1}\parens{Q^{m};\mN} \).
Using a dilation argument if necessary, we may furthermore assume that the restriction of the singular set of those maps to \( \overline{Q^{m}} \) is the dual skeleton of a cubication of \( \overline{Q^{m}} \).
However, as it is constructed above, the map \( g \) only uncrosses the singularities inside \emph{one} cube, not the full set of singularities of a map in \( \Rrig_{m-\floor{sp}-1} \).
Moreover, this procedure comes without any guarantee that the modified map is close to the original one in the \( W^{s,p} \) distance.

As the general constructions are quite involved, the remaining of this section is devoted to some particular cases to explain the main ideas in a more simple setting, allowing for less involved notation and illustrative figures.
We start by presenting the full approximation procedure in \( W^{1,p}\parens{Q^{3};\mN} \) with \( 1 \leq p < 2 \), which corresponds to the case of \emph{line} singularities.
This is the content of Proposition~\ref{prop:uncrossing_lines} below.
We then briefly explain the additional difficulties that arise when moving towards the general case.
For this purpose, we explain how to obtain a smooth construction suited to the full range \( 0 < s < +\infty \), and we also sketch the topological part of our construction to uncross \emph{plane} singularities in \( Q^{3} \).
The proof of Theorem~\ref{theorem:main} in the general case is postponed to Section~\ref{subsect:general_procedure}.

\begin{proposition}
\label{prop:uncrossing_lines}
	Let \( u \in \Rrig_{1}\parens{Q^{3};\mN} \) and \( 1 \leq p < 2 \).
	There exists a sequence \( \parens{u_{n}}_{n \in \bN} \) in \( W^{1,p}\parens{Q^{3};\mN} \) such that \( u_{n} \to u \) in \( W^{1,p}\parens{Q^{3};\mN} \) and such that each \( u_{n} \) is locally Lipschitz outside of a \( 1 \)\=/dimensional Lipschitz submanifold \( \mS_{u_{n}} \) of \( Q^{3} \). 
\end{proposition}

To avoid technicalities and focus on the core of the argument, we have stated Proposition~\ref{prop:uncrossing_lines} with approximating maps being only locally Lipschitz outside of the singular set.
In the proof of the general case of our main result, in Section~\ref{sect:general_case}, we will take care of making the approximating maps smooth and establishing the estimates near the singular set in order to ensure that they belong to the class \( \Rsmooth \).

\begin{proof}
	Since \( u \in \Rrig_{1}\parens{Q^{3};\mN} \), we may assume that its singular set \( \mS_{u} \) is the dual skeleton \( \mT^{1} \) of the \( 1 \)\=/skeleton \( \mK^{1} \) of a cubication of \( \overline{Q^{3}} \) of inradius \( \eta \), for some \( \eta \in \frac{1}{2\bN_{\ast}} \).
	Let \( \mV^{1} \) be the vertical part of \( \mT^{1} \), that is, \( \mV^{1} \) consists of all the lines in \( \mT^{1} \) having directing vector \( (0,0,1) \).
	Let also \( \mV^{1}\trun = \mV^{1} \cap \parens[\big]{\parens{-1,1}^{2} \times \parens{-1+\eta,1}} \) be the vertical singular set \( \mV^{1} \) to which we have truncated the lower extremity. 
	For every \( 0 < \mu < \frac{1}{2} \), we define \( W_{\mu} = \parens{\mV^{1}\trun + Q_{\mu\eta}} \cap Q^{3} \).
	We note that the well \( W_{\mu} \) contains all the crossings of the singular set \( \mT^{1} \).
	The reader may refer to Figure~\ref{fig:well_Q3} for an illustration of the well \( W_{\mu} \) and the singular set \( \mS_{u} \).
	
	We uncross the singularities of \( u \) in two steps.
	The first one, of topological nature, consists in replacing \( u \) in \( W_{\mu} \) by another extension of \( u\restrfun{\partial W_{\mu}} \).
	This extension is constructed in a way that produces a singular set \emph{without} crossings, but comes with no control on the energy of the resulting map.
	The second step, of analytical nature, consists in modifying the map obtained in the first step to obtain a better map with a control on the energy.
	
	\begin{Step}
		\label{step:uncrossing_lines}
		Uncrossing the singularities.
	\end{Step}
	We construct a Lipschitz map \( \upPhi^{\rtop}_{\mu} \colon Q^{3} \to Q^{3} \) such that \( \upPhi^{\rtop}_{\mu} = \id \) outside of \( W_{\mu} \) and \( \parens{\upPhi^{\rtop}_{\mu}}^{-1}\parens{\mT^{1}} \) is a Lipschitz submanifold of \( Q^{3} \).
	Assuming that the map \( \upPhi^{\rtop}_{\mu} \) has been constructed, we explain how to conclude Step~\ref{step:uncrossing_lines}.
	We define the map \( v_{\mu} \colon Q^{3} \setminus \mS_{\mu} \to \mN \) by \( v_{\mu} = u \circ \upPhi^{\rtop}_{\mu} \).
	Here, \( \mS_{\mu} = \parens{\upPhi^{\rtop}_{\mu}}^{-1}\parens{\mT^{1}} \) is the singular set of \( v_{\mu} \), which is a Lipschitz submanifold of \( Q^{3} \) by assumption on \( \upPhi^{\rtop}_{\mu} \).
	Then, the map \( v_{\mu} \) is locally Lipschitz on \( Q^{3} \setminus \mS_{\mu} \), and it coincides with \( u \) outside of \( W_{\mu} \).
	
	We now explain how to construct the map \( \upPhi^{\rtop}_{\mu} \).
	The procedure is illustrated on Figure~\ref{fig:well_Q3}.
	For the part of \( W_{\mu} \) that lies around each line in \( \mV^{1} \), we proceed similarly to what we did on the model case described by Figure~\ref{fig:retraction_g}: We choose a projection point slightly above the line, and we use this point to retract radially the part of \( \overline{W_{\mu}} \) onto the corresponding part of \( \partial W_{\mu} \).
	We note that here, topological operations like closure or boundary are taken inside \( Q^{3} \).
	For instance, \( \partial W_{\mu} \) denotes the boundary of \( W_{\mu} \) in \( Q^{3} \) with respect to the subspace topology.
	This avoids having to systematically take the intersection with \( Q^{3} \) to remove the part of \( \partial W_{\mu} \) that would lie in the boundary of \( Q^{3} \) in \( \bR^{3} \).
	
	Carrying out this construction around each part of \( W_{\mu} \) produces a smooth retraction of \( \overline{W_{\mu}} \) onto \( \partial W_{\mu} \).
	Extending this map by identity outside of \( \overline{W_{\mu}} \) yields a Lipschitz map \( \upPhi^{\rtop}_{\mu} \colon Q^{3} \to Q^{3} \) such that \( \upPhi^{\rtop}_{\mu} = \id \) outside of \( W_{\mu} \).
	
	As \( \mT^{1} \) is a union of line segments which cross only in \( W_{\mu} \), we know that \( \mT^{1} \cap \parens{Q^{3} \setminus W_{\mu}} \) is a Lipschitz submanifold of \( Q^{3} \) with boundary, the latter being the finite set of points \( \mT^{1} \cap \partial W_{\mu} \).
	On the other hand, by construction of \( \upPhi^{\rtop}_{\mu} \), the set \( \parens{\parens{\upPhi^{\rtop}_{\mu}}\restrfun{\overline{W_{\mu}}}}^{-1}\parens{\mT^{1}} \) is a Lipschitz submanifold of \( Q^{3} \) --- actually a set of lines --- also with boundary given by the finite set of points \( \mT^{1} \cap \partial W_{\mu} \).
	Therefore, we conclude that \( \mS_{\mu} \) is a Lipschitz submanifold of \( Q^{3} \) (without boundary), which is depicted on the second cube in Figure~\ref{fig:well_Q3}.
	This finishes to prove that the map \( \upPhi^{\rtop}_{\mu} \) enjoys all the required properties.

	\begin{figure}[p]
		\centering
		\includegraphics[page=4]{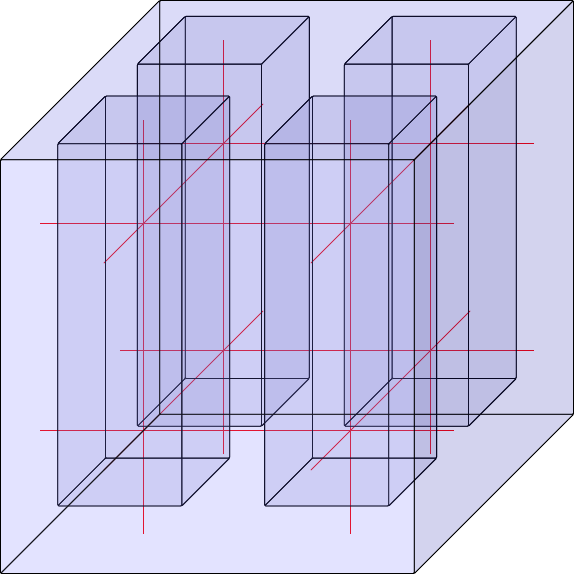}
		\vfill
		\includegraphics[page=5]{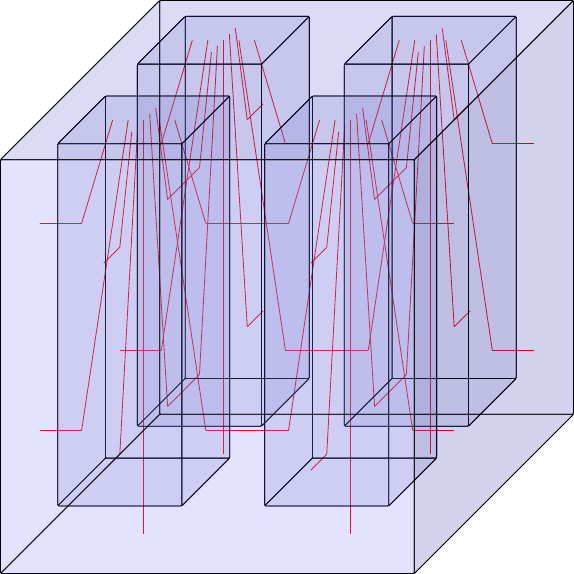}
		\caption{The well \( W_{\mu} \), with singularities before and after uncrossing}
		\label{fig:well_Q3}
	\end{figure}

	\begin{Step}
	\label{step:controlling_energy_lines}
		Controlling the energy.
	\end{Step}
	In the second step, we explain how to modify the map \( v_{\mu} \) in order to obtain a better map \( u_{\mu} \) with controlled energy.
	This relies on a scaling argument.
	For this, the key observation is that, as \( p < 2 \), contracting a Sobolev map to a smaller region decreases its energy in dimension \( 2 \).
	Let \( V_{\mu} = \parens{\mV^{1} + Q_{\mu\eta}} \cap Q^{3} \) be a neighborhood of inradius \( \mu\eta \) of the vertical part of the singular set of \( u \).
	We note that \( W_{\mu} \subset V_{\mu} \) (actually, \( W_{\mu} \) corresponds to \( V_{\mu} \) with its lower part truncated).
	The region \( V_{2\mu} = \parens{\mV^{1} + Q_{2\mu\eta}} \cap Q^{3} \) is a twice larger neighborhood of the vertical part of the singular set of \( u \).
	Given \( 0 < \tau < 1 \), we are going to shrink the values of \( v_{\mu} \) in \( V_{\mu} \) to the small region \( V_{\tau\mu} = \parens{\mV^{1} + Q_{\tau\mu\eta}} \cap Q^{3} \) while keeping \( v_{\mu} \) unchanged outside of \( V_{2\mu} \).
	As explained above, choosing \( \tau \) sufficiently small, we may make the energy of the shrunk map as small as we want on \( V_{\tau\mu} \), hence obtaining a new map with controlled energy regardless of the energy of the extension \( v_{\mu} \) constructed in first instance.
	The region \( V_{2\mu} \setminus V_{\tau\mu} \) serves as a transition region.
	The energy on this region remains under control, since we use the values of \( v_{\mu} \) outside of \( V_{\mu} \), where it coincides with the original map \( u \).
	
	We start with the model case of one vertical rectangle.
	Let \( R_{\mu} = \parens{-\mu\eta,\mu\eta}^{2} \times \parens{-1,1} \).
	Given \( v \in W^{1,p}\parens{Q^{3};\mN} \), we define \( v^{\sh}_{\tau} \in W^{1,p}\parens{Q^{3};\mN} \) by
	\[
		v^{\sh}_{\tau}\parens{x',x_{3}} = 
		\begin{cases}
			v\parens{x',x_{3}} & \text{if \( \parens{x',x_{3}} \in Q^{3} \setminus R_{2\mu} \);} \\
			v\parens{\frac{x'}{\tau},x_{3}} & \text{if \( \parens{x',x_{3}} \in R_{\tau\mu} \);} \\
			v\parens[\Big]{\frac{x'}{\abs{x'}}\parens[\big]{\frac{1}{2-\tau}\parens{\abs{x'}-\tau\mu\eta}+\mu\eta}, x_{3}} & \text{otherwise.}
		\end{cases}
	\]
	Relying on the additivity of the integral and the change of variable theorem, we estimate
	\[
		\int_{R_{2\mu}} \abs{Dv^{\sh}_{\tau}}^{p}
		\lesssim
		\int_{R_{2\mu} \setminus R_{\mu}} \abs{Dv}^{p}
		+
		\tau^{2-p}\int_{R_{\mu}} \abs{Dv}^{p}. 
	\]
	
	We now turn to the modification of our map \( v_{\mu} \).
	Applying the above construction to \( v_{\mu} \) on each rectangle constituting \( V_{2\mu} \), which is nothing else but a translate of \( R_{2\mu} \), we obtain a map \( u_{\mu,\tau} \in W^{1,p}\parens{Q^{3};\mN} \) such that
	\begin{enumerate}[label=(\roman*)]
		\item \( u_{\mu,\tau} \) is locally Lipschitz on \( Q^{3} \setminus \mS_{\mu,\tau} \), where \( \mS_{\mu,\tau} \) is a Lipschitz submanifold of \( Q^{3} \);
		\item \( u_{\mu,\tau} = v_{\mu} = u \) outside of \( V_{2\mu} \);
		\item 
		\[
			\int_{V_{2\mu}} \abs{Du_{\mu,\tau}}^{p}
			\lesssim
			\int_{V_{2\mu} \setminus V_{\mu}} \abs{Dv_{\mu}}^{p}
			+
			\tau^{2-p}\int_{V_{\mu}} \abs{Dv_{\mu}}^{p}. 
		\]
	\end{enumerate}
	Since \( p < 2 \), we may choose \( \tau = \tau_{\mu} \) sufficiently small, depending on \( \mu \), so that
	\begin{equation}
	\label{eq:choice_tau_lines}
		\tau^{2-p}\int_{V_{\mu}} \abs{Dv_{\mu}}^{p}
		\lesssim
		\int_{V_{2\mu} \setminus V_{\mu}} \abs{Dv_{\mu}}^{p}.
	\end{equation}
	We now let \( u_{\mu} = u_{\mu,\tau_{\mu}} \).
	Since \( u_{\mu} = u \) outside of \( V_{2\mu} \), we deduce that
	\[
		\int_{Q^{3}} \abs{Du-Du_{\mu}}^{p}
		\leq
		\int_{V_{2\mu}} \abs{Du-Du_{\mu}}^{p}
		\lesssim
		\int_{V_{2\mu}} \abs{Du}^{p} + \int_{V_{2\mu} \setminus V_{\mu}} \abs{Dv_{\mu}}^{p} + \tau^{2-p}\int_{V_{\mu}} \abs{Dv_{\mu}}^{p}.
	\]
	As \( v_{\mu} = u \) outside of \( V_{\mu} \), we infer from~\eqref{eq:choice_tau_lines} that 
	\[
		\int_{Q^{3}} \abs{Du-Du_{\mu}}^{p}
		\lesssim
		\int_{V_{2\mu}} \abs{Du}^{p}.
	\]
	But \( \abs{V_{\mu}} \to 0 \) as \( \mu \to 0 \), so that Lebesgue's lemma ensures that \( Du_{\mu} \to Du \) in \( L^{p}\parens{Q^{3}} \) as \( \mu \to 0 \).
	On the other hand, since \( \mN \) is compact, we readily have \( u_{\mu} \to u \) in \( L^{p}\parens{Q^{3}} \) as \( \mu \to 0 \).
	Hence, we conclude that \( u_{\mu} \to u \) in \( W^{1,p}\parens{Q^{3}} \) as \( \mu \to 0 \).
	Since \( u_{\mu} \) is locally Lipschitz outside of \( \mS_{\mu,\tau_{\mu}} \), which is a Lipschitz submanifold of \( Q^{3} \), this finishes the proof of the proposition.
	\resetstep
\end{proof}

In the proof of Proposition~\ref{prop:uncrossing_lines} above, the uncrossing map \( \upPhi^{\rtop}_{\mu} \) is only Lipschitz, but not smooth.
This is due to the use of a modified version of the radial retraction of a cube minus its origin to its boundary.
However, in order to obtain a construction compatible with the higher order regularity of Sobolev mappings in the full range \( 0 < s < +\infty \), one needs to work with smooth, and not only merely Lipschitz maps.
To achieve this, one should replace the retraction onto a boundary by a retraction onto a thick region.
To ease the understanding of the general case in Section~\ref{subsect:general_procedure}, we sketch the construction again in the special case of dimension \( m = 3 \) and line singularities.

We keep the same notation as in the proof of Proposition~\ref{prop:uncrossing_lines}, and in particular, we work again with the region \( W_{\mu} \).
We define the two additional, smaller regions \( W_{\mu,\mathrm{in}} \) and \( W_{\mu,\mathrm{out}} \) by \( W_{\mu,\mathrm{in}} = W_{\ulrho\mu} \) and \( W_{\mu,\mathrm{out}} = W_{\olrho\mu} \), where \( 0 < \ulrho < \olrho < 1 \) are fixed.
This way, we have \( W_{\mu,\mathrm{in}} \subset W_{\mu,\mathrm{out}} \subset W_{\mu} \), and all these regions still contain all the crossings of the singular set \( \mT^{1} \).

The starting point is the fact that, for any two open cubes \( Q_{\mathrm{in}} \subsetneq Q_{\mathrm{out}} \) centered at \( 0 \), denoting \( \bR^{3}_{-} = \bR^{2} \times \parens{-\infty,0} \), there exists a smooth diffeomorphism \( \upTheta \colon \bR^{3}_{-}  \to \bR^{3}_{-} \) such that \( \upTheta\parens{\bR^{3}_{-}} \subset \bR^{3}_{-} \setminus Q_{\mathrm{in}} \) and \( \upTheta = \id \) outside of \( Q_{\mathrm{out}} \).
Such a map is obtained by letting \( \upTheta\parens{x} = \lambda\parens{x}x \), where \( \lambda \colon \bR^{3}_{-} \to  \lbrack 1,+\infty\rparen \) is suitably chosen, and satisfies in particular \( \lambda = 1 \) outside of \( Q_{\mathrm{out}} \).

Inserting appropriately scaled copies of \( \upTheta \) around each part of \( W_{\mu} \) and extending by identity outside produces a smooth map \( \upPhi^{\rtop}_{\mu} \colon Q^{3} \to Q^{3} \) such that \( \upPhi^{\rtop}_{\mu} = \id \) outside of \( W_{\mu,\mathrm{out}} \) and \( \upPhi^{\rtop}_{\mu}\parens{Q^{3}} \subset Q^{3} \setminus W_{\mu,\mathrm{in}} \).
Thanks to our refined construction, replacing the retraction onto a boundary by a retraction onto the thick region \( W_{\mu,\mathrm{out}} \setminus W_{\mu,\mathrm{in}} \), we managed to produce a map which is not only Lipschitz but even smooth, while retaining the key feature that its range has to avoid the crossings in the singular set \( \mT^{1} \).
Therefore, \( \parens{\upPhi^{\rtop}_{\mu}}^{-1}\parens{\mT^{1}} \) is a smooth submanifold of \( Q^{3} \), as we needed.
We shall build upon this idea, in combination with the constructions sketched below for handling higher dimensional singularities, to prove Theorem~\ref{theorem:main} in full generality in the whole range \( 0 < s  < +\infty \).

\bigskip

We now turn to the case of the density of the class \( \Rsmooth_{2}\parens{Q^{3};\mN} \), where the maps have plane singularities.
Compared to the case of line singularities treated previously, the first topological step consisting in uncrossing the singularities features an additional difficulty, that we explain in this subsection in an informal way, with the help of figures.
The precise construction of the topological step, as well as the analytical step in which we improve the construction with a control on the energy and which relies on the same scaling argument as for line singularities, are postponed to Section~\ref{subsect:general_procedure}, where we explain precisely the general tools needed to prove Theorem~\ref{theorem:main}.

Consider a singular set \( \mT^{2} \) for a map \( u \) in \( \Rrig_{2}(Q^{3};\mN) \), given by the dual skeleton of the \( 0 \)\=/skeleton \( \mK^{0} \) of a cubication of \( \overline{Q^{3}} \) having inradius \( \eta \in \frac{1}{2\bN_{\ast}} \).
As previously, we let \( \mV^{2} \) denote the vertical part of \( \mT^{2} \), that is, the union of all hyperplanes which constitute \( \mT^{2} \) whose associated vector space contains \( e_{3} \).
This set is made of two unions of parallel planes: the set \( \mT_{1,3} \) consisting of all the planes in \( \mT^{2} \) whose associated vector space is spanned by \( e_{1} \) and \( e_{3} \), and the set \( \mT_{2,3} \) consisting of all the planes in \( \mT^{2} \) whose associated vector space is spanned by \( e_{2} \) and \( e_{3} \).
We also let \( \mV^{2}\trun = \mV^{2} \cap \parens[\big]{\parens{-1,1}^{2} \times \parens{-1+\eta,1}} \) be the truncated version of \( \mV^{2} \).

As previously, given \( 0 < \mu < \frac{1}{2} \), we consider \( W_{\mu} = \parens{\mV^{2}\trun + Q_{\mu\eta}} \cap Q^{3} \) a well around \( \mV^{2}\trun \).
We first uncross the singularities in \( W_{\mu} \) as follows.
We start with a model construction to uncross two families of parallel planes.
We observe that the construction carried out for lines in \( Q^{3} \) in the proof of Proposition~\ref{prop:uncrossing_lines} may also be applied to lines in \( Q^{2} \).
Indeed, it suffices to perform a radial projection around a point outside \( Q^{2} \) in order to retract \( \overline{Q^{2}} \) onto all its edges except one.
This construction can then be applied to uncross two planes (or, more precisely, portions of planes).
Assume that one wants to uncross the singularities around the vertical portion of plane \( \mP = \set{x \in Q^{3}\st x_{1} = 0} \).
Consider the line segment \( \mL = \set{0} \times \parens{-1,1} \times \set{1+\varepsilon} \), which is a line segment subparallel to \( \mP \) and lying slightly above \( \mP \).
For every plane orthogonal to \( \mL \) determined by \( x_{2} = t \) with \( -1 < t < 1 \), one performs the \( 2 \)\=/dimensional uncrossing procedure in this plane with respect to the unique point of \( \mL \) lying in the plane.
Otherwise stated, one proceeds to a radial projection around a line segment in the \( e_{2} \) direction lying slightly above \( \mP \), viewing the second coordinate variable as a dummy variable.
This allows to uncross \( \mP \) from other planes in horizontal position.
The procedure is illustrated in Figure~\ref{fig:uncrossing_plane_one_direction}: the vertical plane around which the well has been dug is uncrossed from the horizontal plane, and both vertical planes are left unchanged.

We may then elaborate on this idea to uncross all singularities in \( W_{\mu} \), as described in Figure~\ref{fig:uncrossing_plane_vertical}.
On the parts of \( W_{\mu} \) that do not contain a crossing between two vertical planes (the four darkest parts around the central one in Figure~\ref{fig:uncrossing_plane_vertical}), we insert a copy of the construction described just above, as in Figure~\ref{fig:uncrossing_plane_one_direction}.
We note that the constructions are compatible on the region where two different parts touch --- which is a union of vertical line segments --- since they coincide with the identity there.
On the parts of \( W_{\mu} \) around the crossing between orthogonal vertical planes (the central part in Figure~\ref{fig:uncrossing_plane_vertical}), we finish the construction of our extension by using the radial projection from a point slightly above the crossing, as we did for line singularities.
The resulting effect of these glued constructions is to remove all the crossings between horizontal and vertical planes; see Figure~\ref{fig:uncrossing_plane_vertical}.
However, unlike in the case of line singularities, we are not done yet, since there still are crossings between orthogonal vertical planes to remove.

\begin{figure}[ht]
	\centering
	~
	\hfill
	\includegraphics[page=6]{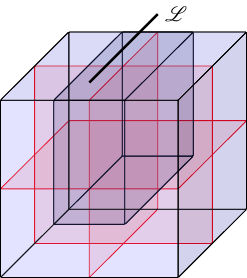}
	\hfill
	\includegraphics[page=7]{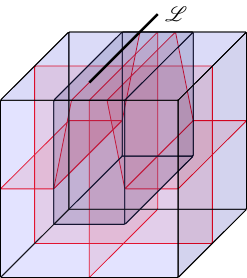}
	\hfill
	~
	\caption{Uncrossing plane singularities in one direction}
	\label{fig:uncrossing_plane_one_direction}
\end{figure}

\begin{figure}[ht]
	\centering
	\includegraphics[page=8]{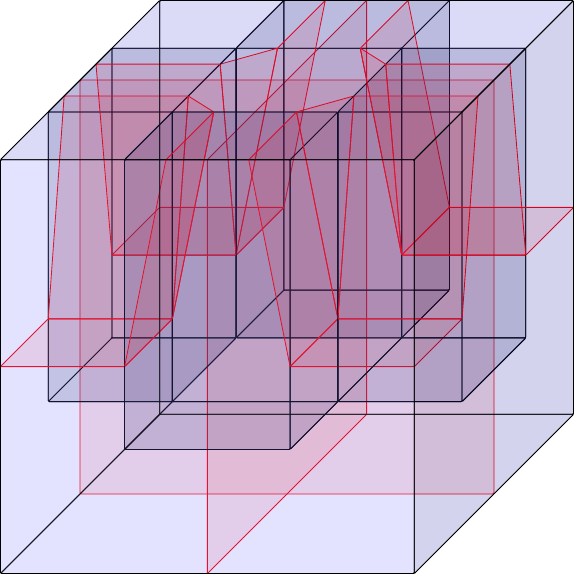}
	\caption{Uncrossing plane singularities around all vertical planes}
	\label{fig:uncrossing_plane_vertical}
\end{figure}

For this purpose, we use a well in another direction.
We consider the truncated set of parallel hyperplanes \( \mH\trun^{2} = \mT^{1,3} \cap \parens[\big]{\parens{-1+\eta,1} \times \parens{-1,1}^{2}} \), and the well \( H_{\mu} = \parens{\mH\trun^{2} + Q_{\rho\mu\eta}} \cap Q^{3} \), where \( 0 < \rho < 1 \) is chosen sufficiently small so that \( H_{\mu} \) intersects only the planes in the singular set that have not yet been uncrossed.
We note that \( H_{\mu} \) contains all the remaining singularities.
We then insert a rotated copy of the construction illustrated in Figure~\ref{fig:uncrossing_plane_one_direction} in each part of the well \( H_{\mu} \) around a plane constituting \( \mH\trun^{2} \).
The procedure is illustrated in Figure~\ref{fig:uncrossing_plane_horizontal} in the case where there is only one plane in each direction.
At the end of this step, the crossings between orthogonal vertical planes have been removed, and therefore no crossings remain.

\begin{figure}[ht]
	\centering
	\includegraphics[page=9]{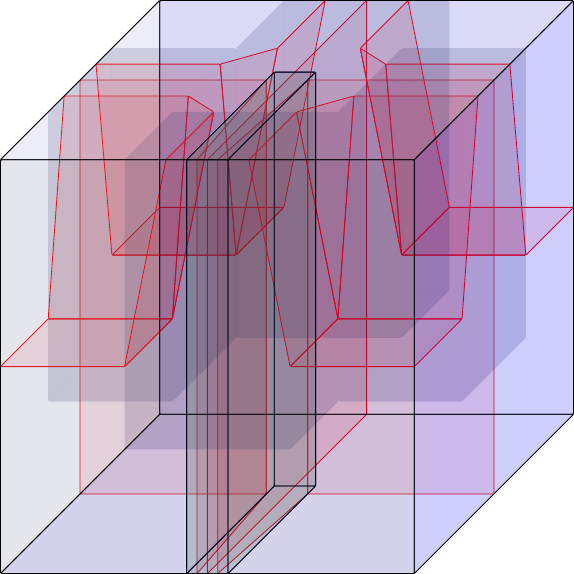}
	\caption{Uncrossing plane singularities between vertical planes}
	\label{fig:uncrossing_plane_horizontal}
\end{figure}

This concludes our informal presentation of some particular cases of crossings removal.
In the next section, we introduce the general version of the two main tools that have been presented here: the topological construction to remove crossings, and the analytical procedure to control the energy on the modified region.
These two tools are the key ingredients in the proof of our main result.
Concerning the second one, we use the shrinking construction introduced by Bousquet, Ponce, and Van Schaftingen~\cite[Section~8]{BousquetPonceVanSchaftingen2015}; see also~\cite[Section~7]{Detaille2023} for the fractional order setting.
For the first one, however, we need to perform an \emph{ad hoc} construction, suited for our purposes.
This construction is nevertheless very similar to the thickening procedure introduced in~\cite[Section~4]{BousquetPonceVanSchaftingen2015}.
As we have seen in our last example with plane singularities, the crossings removal procedure may involve gluing building blocks in various dimensions and also combining crossings removal procedures in different directions to get rid of all the existing crossings.

\subsection{The general crossings removal procedure}
\label{subsect:general_procedure}

We now explain how to prove our main result, Theorem~\ref{theorem:main}, in the general case.
The argument follows the same two steps as in Proposition~\ref{prop:uncrossing_lines}: First, we uncross the singularities through a topological procedure, and then we rely on an analytical argument to obtain a control on the energy.

We start by considering the first topological step.
This is handled by the following proposition.

\begin{proposition}
\label{prop:main_topological_tool}
	Let \( \ell \in \{0,\dots,m-2\} \) and let \( \mT^{\ell^{\ast}} \) be the dual skeleton of the \( \ell \)\=/skeleton \( \mK^{\ell} \) of a cubication \( \mK^{m} \) of \( \overline{Q^{m}} \) of inradius \( \eta > 0 \).
	For every \( 0 < \mu < 1 \), there exists a smooth local diffeomorphism \( \upPhi \colon Q^{m} \to Q^{m} \) such that
	\begin{enumerate}[label=(\roman*)]
		\item\label{item:main_top_tool_sing_set} \( \mS^{\ell^{\ast}} = \upPhi^{-1}\parens{\mT^{\ell^{\ast}}} \) is a smooth \( \ell^{\ast} \)\=/dimensional submanifold of \( Q^{m} \);
		\item\label{item:main_top_tool_id} \( \upPhi = \id \) outside of \( \mT^{\ell^{\ast}} + Q_{\mu\eta} \).
	\end{enumerate}
	Moreover, \( \upPhi \) can be extended to a smooth local diffeomorphism on a slightly larger open set \( \omega \subset \bR^{m} \) such that \( Q^{m} \Subset \omega \).
\end{proposition}

The proof of Proposition~\ref{prop:main_topological_tool} is similar in its spirit to the \emph{thickening} construction; see~\cite[Section~4]{BousquetPonceVanSchaftingen2015}.
However, to have a tool suited for our purposes here, we cannot re-use thickening as such, and we need to proceed to a quite different construction.
We also note that our restriction \( \ell \leq m-2 \) excludes the case \( \ell = m-1 \), where \( \ell^{\ast} = 0 \), and hence the singular set would have been made of points.
But in this case, the classes \( \Rrig \), \( \Rclas \), and \( \Rsmooth \) all coincide, so that Theorem~\ref{theorem:main} is already contained in Bethuel's theorem and its counterpart for arbitrary \( s \), and requires therefore no additional argument.

\begin{proof}
	As explained in the last example of Section~\ref{subsect:part_cases}, the general uncrossing procedure requires to perform successive uncrossing steps in various directions. 
	\begin{Step}
		Uncrossing singularities in a vertical well.
	\end{Step}
	We let \( \mV \) be the part of \( \mT^{\ell^{\ast}} \) consisting only of \( \ell^{\ast} \)\=/planes that have the vertical vector \( e_{m} \) in their associated vector space.
	We also consider the truncated set of planes \( \mV_{\trun} = \mV \cap \parens{\parens{-1,1}^{m-1} \times \parens{-1+\eta,1}} \). 
	Finally, we let \( W_{\mu} = \parens{\mV_{\trun} + Q_{\mu\eta}} \cap Q^{m} \) be a well around \( \mV_{\trun} \).
	The well \( W_{\mu/2} \) is defined accordingly.
	We note that \( W_{\mu/2} \) contains all the crossings that involve at least one non vertical \( \ell^{\ast} \)\=/plane, i.e., a plane not in \( \mV \). 
	
	Let 
	\[
		\frac{\mu}{2} < \ulrho_{\ell^{\ast}-1} < \olrho_{\ell^{\ast}-1} < \dotsb < \ulrho_{0} < \olrho_{0} < \mu.
	\]
	We consider \( \mE^{\ell^{\ast}-1} = \mV \cap \parens{\parens{-1,1}^{m-1} \times \set{1}} \) the intersection of \( \mV \) with the top face of \( Q^{m} \). We note that \( \mE^{\ell^{\ast}-1} \) is an \( \parens{\ell^{\ast}-1} \)\=/skeleton.
	For every \( d \in \set{0,\dotsc,\ell^{\ast}-1} \), we define the rectangles
	\begin{gather*}
		R^{d} = \parens{-\mu\eta,\mu\eta}^{m-1-d} \times \parens{-\parens{1-\olrho_{d}}\eta,\parens{1-\olrho_{d}}\eta}^{d} \times \parens{-1+\eta-\mu\eta,1}\text{,} \\
		R^{d}_{\mathrm{out}} = \parens{-\olrho_{d}\eta,\olrho_{d}\eta}^{m-1-d} \times \parens{-\parens{1-\olrho_{d}}\eta,\parens{1-\olrho_{d}}\eta}^{d} \times \parens{-1+\eta-\olrho_{d}\eta,1}\text{,}
		\intertext{and}
		R^{d}_{\mathrm{in}} = \parens{-\ulrho_{d}\eta,\ulrho_{d}\eta}^{m-1-d} \times \parens{-\parens{1-\olrho_{d}}\eta,\parens{1-\olrho_{d}}\eta}^{d} \times \parens{-1+\eta-\ulrho_{d}\eta,1}.
	\end{gather*}
	Given a \( d \)\=/face \( \sigma^{d} \in E^{d} \), we let \( R_{\sigma^{d}} \) be the rotated copy of \( R^{d} \) positioned so that \( \sigma^{d} \) corresponds to \( \set{0}^{m-d-1} \times \parens{-1,1}^{d} \times \set{1} \).
	This way, we note that \( R_{\sigma^{d}} \subset W_{\mu} \) for every \( d \in \set{0,\dotsc,\ell^{\ast}-1} \) and every \( \sigma^{d} \in E^{d} \), and that actually \( W_{\mu} \) is made of the union of all such \( R_{\sigma^{d}} \).
	We define similarly \( R_{\sigma^{d},\mathrm{in}} \) and \( R_{\sigma^{d},\mathrm{out}} \).
	
	We use as a tool the following construction from~\cite[Proposition~4.3]{BousquetPonceVanSchaftingen2015}:
	There exists a smooth local diffeomorphism \( \upTheta_{d} \colon Q^{d}_{\mu\eta} \setminus \set{0} \to Q^{d}_{\mu\eta} \) such that 
	\begin{enumerate}[label=(\roman*)]
		\item\label{item:Thetad_geometry} \( \upTheta_{d}\parens{Q^{d}_{\mu\eta} \setminus \set{0}} \subset Q^{d}_{\mu\eta} \setminus Q^{d}_{\ulrho_{d}\eta} \);
		\item\label{item:Thetad_id} \( \upTheta_{d} = \id \) outside of \( Q^{d}_{\olrho_{d}\eta} \).
	\end{enumerate}
	The map \( \upTheta_{d} \) is constructed by letting \( \upTheta_{d}\parens{x} = \lambda\parens{x}x \) for some well-chosen smooth map \( \lambda \colon Q^{d}_{\mu\eta} \setminus \set{0} \to \lbrack 1,+\infty\rparen \) such that \( \lambda = 1 \) outside of \( Q^{d}_{\olrho_{d}\eta} \).
	We focus our attention to the restriction of \( \upTheta_{d} \) to the lower part of \( Q^{d}_{\mu\eta} \), slightly below \( \set{0} \).
	After a suitable distortion of \( Q^{d}_{\mu\eta} \) and addition of dummy variables, this yields a smooth local diffeomorphism \( \upPsi_{d} \colon R^{d} \to R^{d} \) such that 
	\begin{enumerate}[label=(\roman*)]
		\item\label{item:Psid_geometry} \( \upPsi_{d}\parens{R^{d}} \subset R^{d} \setminus R^{d}_{\mathrm{in}} \);
		\item\label{item:Psid_id} \( \upPsi_{d} = \id \) outside of \( R^{d}_{\mathrm{out}} \).
	\end{enumerate}
	Let \( \upPsi_{\sigma^{d}} \) be the map obtained by transporting isometrically \( \upPsi_{d} \) to \( R_{\sigma^{d}} \), and define \( \upPhi_{d}\parens{x} = \upPsi_{\sigma^{d}}\parens{x} \) if \( x \in R_{\sigma^{d}} \).
	We note that this is well defined.
	Indeed, if \( x \in R_{\sigma_{1}^{d}} \cap R_{\sigma_{2}^{d}} \), then \( x \) is outside of \( R_{\sigma_{1}^{d},\mathrm{out}} \) and \( R_{\sigma_{2}^{d},\mathrm{out}} \), which implies that \( \upPsi_{\sigma_{1}^{d}}\parens{x} = x = \upPsi_{\sigma_{2}^{d}}\parens{x} \).
	
	We readily observe that the map \( \upPhi_{d} \) can be smoothly extended by identity to \( Q^{m} \setminus \bigcup\limits_{\substack{\sigma^{l} \in E^{l} \\ l \in \set{0,\dotsc,d-1}}} R_{\sigma^{l},\mathrm{in}} \).
	In particular, this yields \( \upPhi_{d} = \id \) outside of \( W_{\mu} \).
	Moreover, \( \upPhi_{d} \) has the property that it maps \( \bigcup\limits_{\sigma^{d} \in E^{d}} R_{\sigma^{d}} \) outside of \( \bigcup\limits_{\sigma^{d} \in E^{d}} R_{\sigma^{d},\mathrm{in}} \).
	By induction, this implies that the composition \( \upPhi_{\ver} = \upPhi_{\ell^{\ast}-1} \circ \dotsb \circ \upPhi_{0} \) is a well-defined smooth local diffeomorphism and maps \( Q^{m} \) outside of \( W_{\mu/2} \).
	We note importantly that the well-defined character of the map relies on the fact that, although \( \upPhi_{d} \) is not defined on the whole \( Q^{m} \), it is nevertheless defined on the range of the composition of the previous maps in the induction process.
	In particular, \( \upPhi_{\ver}^{-1}\parens{\mT^{\ell^{\ast}}} \) is a finite union of smooth \( \ell^{\ast} \)\=/dimensional submanifolds of \( Q^{m} \), and as \( W_{\mu/2} \) contains all the crossings between \( \ell^{\ast} \)\=/planes in \( \mT^{\ell^{\ast}} \) involving at least one non vertical one, we deduce that the only submanifolds in \( \upPhi_{\ver}^{-1}\parens{\mT^{\ell^{\ast}}} \) that intersect correspond to inverse images of vertical \( \ell^{\ast} \)\=/planes.
	Finally, since the building blocks \( \upTheta_{d} \) have the form \( \upTheta_{d}\parens{x} = \lambda\parens{x}x \), we also find that \( \upPhi_{\ver}^{-1}\parens{\mV} = \mV \).
	
	\begin{Step}
		Uncrossing vertical planes.
	\end{Step}
	It remains to remove the crossings between planes in \( \mV \).
	For this purpose, we choose another --- non vertical --- direction, and we rotate \( Q^{m} \) to make it correspond to the vertical one.
	We then repeat the exact same construction as in the first step, except that we replace \( W_{\mu} \) by \( W_{\rho\mu} \) for some \( 0 < \rho < 1 \) so small that \( W_{\rho\mu} \) does not intersect the inverse images under \( \upPhi_{\ver} \) of \( \ell^{\ast} \)\=/planes of \( \mT^{\ell^{\ast}} \setminus \mV \).
	The construction should then be modified accordingly, adding the scaling \( \rho \) wherever necessary, and this yields another smooth local diffeomorphism \( \upPhi_{\hor} \colon Q^{m} \to Q^{m} \) that coincides with the identity outside of \( W_{\rho\mu} \) and such that \( \upPhi_{\hor}^{-1}\parens{\mV} \) is a finite union of smooth \( \ell^{\ast} \)\=/dimensional submanifolds of \( Q^{m} \).
	Moreover, only the inverse images coming from planes in the new vertical direction may still cross.
	Therefore, the map \( \upPhi_{\ver} \circ \upPhi_{\hor} \) is a smooth local diffeomorphism that coincides with the identity outside of \( \mT^{\ell^{\ast}} + Q_{\mu\eta} \) and such that \( \parens{\upPhi_{\ver} \circ \upPhi_{\hor}}^{-1}\parens{\mT^{\ell^{\ast}}} \) is a finite union of smooth submanifolds of \( Q^{m} \), and only the parts coming from \( \ell^{\ast} \)\=/planes aligned with the two chosen directions may still cross.
	
	We pursue this procedure, choosing each time a new direction to be the vertical one, until no crossing remains.
	This yields the desired map \( \upPhi \).
	
	Moreover, it is readily observed from our construction, since each building block could have been defined on a slightly larger set, that \( \upPhi \) may be extended to a smooth local diffeomorphism defined on a slightly larger set.
	\resetstep
\end{proof}

We now turn to the analytical step.
This relies on the \emph{shrinking} construction, which has been introduced in~\cite[Section~8]{BousquetPonceVanSchaftingen2015}; see also~\cite[Section~7]{Detaille2023} for the fractional order setting.

\begin{proposition}
\label{prop:main_analytical_tool}
	Let \( \ell \in \{0,\dots,m-1\} \), \( 0 < \mu < \frac{1}{2} \), \( 0 < \tau < \frac{1}{2} \), \( \varepsilon > 0 \),	\( \mK^{m} \) be a cubication in \( \bR^{m} \) of radius \( \eta > 0 \), and \( \mT^{\ell^{\ast}} \) be the dual skeleton of \( \mK^{\ell} \).
	If \( \ell + 1 > sp \), then there exists a smooth local diffeomorphism \( \upPhi \colon \bR^{m} \to \bR^{m} \) satisfying \( \upPhi\parens{\sigma^{m}} \subset \sigma^{m} \) for every \( \sigma^{m} \in K^{m} \) and such that, for every \( u \in W^{s,p}\parens{\mK^{m}} \) and every \( v \in W^{s,p}\parens{\mK^{m}} \) such that \( u = v \) on the complement of \( \mT^{\ell^{\ast}}+Q^{m}_{\mu\eta} \), we have \( u \circ \upPhi \in W^{s,p}\parens{\mK^{m}} \), and moreover, there exists a constant \( C > 0 \) depending on \( m \), \( s \), and \( p \) such that
	\begin{enumerate}[label=(\roman*)]
		\item\label{item:below_main_shrinking_sle1} if \( 0 < s < 1 \), then 
		\[
		\lvert u\circ\upPhi - v \rvert_{W^{s,p}\parens{\mK^{m}}} 
		\leq
		C\parens[\Big]{\lvert v \rvert_{W^{s,p}\parens{\mK^{m} \cap \parens{\mT^{\ell^{\ast}}+Q^{m}_{2\mu\eta}}}} + \parens{\mu\eta}^{-s}\lVert v \rVert_{L^{p}\parens{\mK^{m} \cap \parens{\mT^{\ell^{\ast}}+Q^{m}_{2\mu\eta}}}}} + \varepsilon;
		\]
		\item\label{item:below_main_shrinking_sgeq1_integer} if \( s \geq 1 \), then for every \( j \in \{1,\dots,k\} \),
		\[
		\lVert D^{j}\parens{u\circ\upPhi} - D^{j}v \rVert_{L^{p}\parens{\mK^{m}}}
		\leq 
		C\sum_{i=1}^{j} \parens{\mu\eta}^{i-j}\lVert D^{i}v \rVert_{L^{p}\parens{\mK^{m} \cap \parens{\mT^{\ell^{\ast}}+Q^{m}_{2\mu\eta}}}} + \varepsilon;
		\]
		\item\label{item:below_main_shrinking_sgeq1_frac} if \( s \geq 1 \) and \( \sigma \neq 0 \), then for every \( j \in \{1,\dots,k\} \),
		\begin{multline*}
		\lvert D^{j}\parens{u\circ\upPhi} - D^{j}v \rvert_{W^{\sigma,p}\parens{\mK^{m}}} \\
		\leq
		C\sum_{i=1}^{j}\parens[\Big]{\parens{\mu\eta}^{i-j-\sigma}\lVert D^{i}v \rVert_{L^{p}\parens{\mK^{m} \cap \parens{\mT^{\ell^{\ast}}+Q^{m}_{2\mu\eta}}}} + \parens{\mu\eta}^{i-j}\lvert D^{i}v \rvert_{W^{\sigma,p}\parens{\mK^{m} \cap \parens{\mT^{\ell^{\ast}}+Q^{m}_{2\mu\eta}}}}} + \varepsilon;
		\end{multline*}
		\item\label{item:below_main_shrinking_all} for every \( 0 < s < +\infty \),
		\[
		\lVert u \circ \upPhi - v \rVert_{L^{p}\parens{\mK^{m}}}
		\leq 
		C\lVert v \rVert_{L^{p}\parens{\mK^{m} \cap \parens{\mT^{\ell^{\ast}}+Q^{m}_{2\mu\eta}}}} + \varepsilon.
		\]
	\end{enumerate}
\end{proposition}

Proposition~\ref{prop:main_analytical_tool} is obtained from~\cite[Proposition~7.1]{Detaille2023} by choosing \( \tau \) sufficiently small, depending on \( u \) and \( v \), as explained below the proposition.

Having at hand Propositions~\ref{prop:main_topological_tool} and~\ref{prop:main_analytical_tool}, we are ready to perform the topological and analytical steps of our construction.
However, before proving Theorem~\ref{theorem:main}, we need one last technical tool.
Indeed, our proof involves composing the map \( u \in \Rrig \subset \Rclas \) we want to approximate with the maps provided by Propositions~\ref{prop:main_topological_tool} and~\ref{prop:main_analytical_tool}.
The following lemma ensures that the class \( \Rclas \) is stable through composition with a local diffeomorphism.

\begin{lemma}
\label{lemma:Rclas_local_diffeomorphism}
	Let \( \varepsilon > 0 \), and let \( \upPhi \colon Q_{1+\varepsilon} \to \bR^{m} \) be a local diffeomorphism such that \( \upPhi\parens{Q^{m}} \subset Q^{m} \).
	For every \( u \in \Rclas_{i}\parens{Q^{m}} \), we have that \( u \circ \upPhi \in \Rclas_{i}\parens{Q^{m}} \).
\end{lemma}
\begin{proof}
	Let \( \mS \) denote the singular set of \( u \).
	We may assume that \( \mS \neq \varnothing \), otherwise the proof is trivial.
	We may also assume that \( \upPhi \) is actually a local diffeomorphism defined on the whole \( \bR^{m} \).
	Indeed, if this is not the case, we consider a diffeomorphism \( \upPsi \colon \bR^{m} \to Q_{1+\varepsilon} \) such that \( \upPsi = \id \) on \( Q^{m} \), and we replace \( \upPhi \) by \( \upPhi \circ \upPsi \), which is a local diffeomorphism on \( \bR^{m} \) and coincides with \( \upPhi \) on \( Q^{m} \).
	We note importantly that, if \( \upPhi \) has the additional property that \( \upPhi^{-1}\parens{\mS} \) is a closedly embedded submanifold of \( Q_{1+\varepsilon} \) (for the relative topology), then \( \parens{\upPhi \circ \upPsi}^{-1}\parens{\mS} \) is a closedly embedded submanifold of \( \bR^{m} \).
	This will be important in the sequel, when working with constructions to produce maps in the class \( \Rsmooth \).
	
	Since \( \upPhi \) is a local diffeomorphism, the map \( u \circ \upPhi \) is smooth on \(  Q^{m} \setminus \tilde{\mS} \), where \( \tilde{\mS} = \upPhi^{-1}\parens{\mS} \) is a finite union of smooth \( i \)\=/dimensional submanifolds of \( \bR^{m} \).
	Moreover, if \( u \) extends smoothly on \( U \setminus \mS \) for some open set \( U \subset \bR^{m} \) satisfying \( Q^{m} \Subset U \), then \( u \circ \upPhi \) extends smoothly on \( \upPhi^{-1}\parens{U} \setminus \tilde{\mS} \), and \( \upPhi^{-1}\parens{U} \) is an open subset of \( \bR^{m} \) satisfying \( Q^{m} \Subset \upPhi^{-1}\parens{U} \).
	It therefore remains to prove the estimates on the derivatives of \( u \circ \upPhi \).
	
	For this purpose, we first note that, as \( \upPhi \) is defined on the whole \( \bR^{m} \), it has bounded derivatives on \( Q^{m} \).
	Therefore, the Faà di Bruno formula ensures that, for every \( x \in Q^{m} \) and \( j \in \bN_{\ast} \),
	\[
		\abs{D^{j}\parens{u \circ \upPhi}\parens{x}}
		\lesssim
		\sum_{t=1}^{j} \abs{D^{t}u\parens{\upPhi\parens{x}}}
		\lesssim
		\sum_{t=1}^{j}\frac{1}{\dist{\parens{\upPhi\parens{x},\mS}}^{t}}
		\lesssim
		\frac{1}{\dist{\parens{\upPhi\parens{x},\mS}}^{j}}.
	\]
	We conclude the proof by showing that \( \dist{\parens{\upPhi\parens{x},\mS}} \gtrsim \dist{\parens{x,\tilde{\mS}}} \) for every \( x \in Q^{m} \).
	
	For this purpose, we first note that, by a compactness argument, there exists \( \delta > 0 \) such that, for every \( x \in Q^{m} \), the restriction of \( \upPhi \) to \( \upPhi^{-1}\parens{B_{\delta}\parens{\upPhi\parens{x}}} \) is a diffeomorphism onto \( B_{\delta}\parens{\upPhi\parens{x}} \).
	Taking \( \delta \) smaller if necessary, this implies in particular that 
	\begin{equation}
	\label{eq:inverse_lipschitz_local_diffeomorphism}
		\abs{x-y} \lesssim \abs{\upPhi\parens{x}-\upPhi\parens{y}}
		\quad
		\text{whenever \( \abs{\upPhi\parens{x} - \upPhi\parens{y}} < \delta \).}
	\end{equation}
	It suffices to show that \( \dist{\parens{\upPhi\parens{x},\mS}} \gtrsim \dist{\parens{x,\tilde{\mS}}} \) whenever \( x \in Q^{m} \) is such that \( \upPhi\parens{x} \) is sufficiently close to \( \mS \).
	Hence, let \( x \in Q^{m} \) be such that \( \dist{\parens{\upPhi\parens{x},\mS}} < \delta \).
	Let \( z \in \mS \) be such that \( \abs{\upPhi\parens{x}-z} = \dist{\parens{\upPhi\parens{x},\mS}} \).
	In particular, there exists \( y \in \upPhi^{-1}\parens{B_{\delta}\parens{\upPhi\parens{x}}} \) such that \( \upPhi\parens{y} = z \).
	With this choice, we have \( y \in \upPhi^{-1}\parens{\mS} = \tilde{\mS} \) and therefore, due to~\eqref{eq:inverse_lipschitz_local_diffeomorphism},
	\[
		\dist{\parens{x,\tilde{\mS}}}
		\leq
		\abs{x-y}
		\lesssim
		\abs{\upPhi\parens{x}-\upPhi\parens{y}}
		=
		\dist{\parens{\upPhi\parens{x},\mS}}.
	\]
	This concludes the proof of the lemma.
\end{proof}

\begin{proof}[Proof of Theorem~\ref{theorem:main}]
	Since the more rigid class \( \Rrig_{m-[sp]-1}\parens{Q^{m};\mN} \) is dense in \( W^{s,p}\parens{Q^{m};\mN} \), it suffices to consider \( u \in \Rrig_{m-\floor{sp}-1}\parens{Q^{m};\mN} \) and to show that it can be approximated by maps in the uncrossed class \( \Rsmooth_{m-\floor{sp}-1}\parens{Q^{m};\mN} \).
	Let \( \ell = \floor{sp} \), and let \( \mT^{\ell^{\ast}} \) be the dual skeleton of the \( \ell \)\=/skeleton \( \mK^{\ell} \) of a cubication \( \mK^{m} \) of radius \( \eta > 0 \) of \( \overline{Q^{m}} \), chosen so that \( \mT^{\ell^{\ast}} \) coincides with the singular set of \( u \).
	Recall that, as already explained, we may limit ourselves to consider maps such that their singular set is placed like this.
	Also, recall that we may assume that \( \ell \leq m-2 \), as if \( \ell = m-1 \) there is nothing to prove.
	
	For the sake of conciseness, we let \( \mA_{\mu} = Q^{m} \cap \parens{\mT^{\ell^{\ast}}+Q_{2\mu\eta}} \).
	Given \( 0 < \mu < \frac{1}{2} \), we start by applying Proposition~\ref{prop:main_topological_tool} to obtain a map \( \upPhi^{\rtop}_{\mu} \colon Q^{m} \to Q^{m} \) such that, defining \( u^{\rtop}_{\mu} = u \circ \upPhi^{\rtop}_{\mu} \), we have that 
	\begin{enumerate}[label=(\roman*)]
		\item\label{item:utop_sing_set} \( \mS^{\rtop}_{\mu} = \parens{\upPhi^{\rtop}_{\mu}}^{-1}\parens{\mT^{\ell^{\ast}}} \) is a smooth \( \ell^{\ast} \)\=/dimensional submanifold of \( Q^{m}\);
		\item\label{item:utop_eq_u} \( u^{\rtop}_{\mu} = u \) outside of \( \mT^{\ell^{\ast}} + Q_{\mu\eta} \);
		\item\label{item:utop_sobolev} \( u^{\rtop}_{\mu} \in W^{s,p}\parens{Q^{m}} \).
	\end{enumerate}
	Item~\ref{item:utop_sobolev} above is a consequence of the fact that \( u^{\rtop}_{\mu} \in \Rclas_{m-\floor{sp}-1}\parens{Q^{m};\mN} \), as \( u \in \Rclas_{m-\floor{sp}-1}\parens{Q^{m};\mN} \) and using Lemma~\ref{lemma:Rclas_local_diffeomorphism}.

	Since \( \ell+1 = \floor{sp}+1 > sp \) and thanks to~\ref{item:utop_eq_u} and~\ref{item:utop_sobolev} above, we may now invoke Proposition~\ref{prop:main_analytical_tool} on \( u^{\rtop}_{\mu} \) and \( u \), with \( \varepsilon = \mu \), to deduce the existence of a smooth local diffeomorphism \( \upPhi^{\sh}_{\mu} \colon Q^{m} \to Q^{m} \) such that, letting \( u^{\sh}_{\mu} = u^{\rtop}_{\mu} \circ \upPhi^{\sh}_{\mu} \), we have that \( u^{\sh}_{\mu} \in W^{s,p}\parens{Q^{m}} \) with
	\begin{enumerate}[label=(\roman*)]
		\item\label{item:ush_sle1} if \( 0 < s < 1 \), then 
		\[
		\lvert u^{\sh}_{\mu} - u \rvert_{W^{s,p}\parens{Q^{m}}} 
		\lesssim
		\lvert u \rvert_{W^{s,p}\parens{\mA_{\mu}}} + \parens{\mu\eta}^{-s}\lVert u \rVert_{L^{p}\parens{\mA_{\mu}}} + \mu\text{;}
		\]
		\item\label{item:ush_sgeq1_integer} if \( s \geq 1 \), then for every \( j \in \{1,\dots,k\} \),
		\[
		\lVert D^{j}u^{\sh}_{\mu} - D^{j}u \rVert_{L^{p}\parens{Q^{m}}}
		\lesssim
		\sum_{i=1}^{j} \parens{\mu\eta}^{i-j}\lVert D^{i}u \rVert_{L^{p}\parens{\mA_{\mu}}} + \mu\text{;}
		\]
		\item\label{item:ush_sgeq1_frac} if \( s \geq 1 \) and \( \sigma \neq 0 \), then for every \( j \in \{1,\dots,k\} \),
		\[
		\lvert D^{j}u^{\sh}_{\mu} - D^{j}u \rvert_{W^{\sigma,p}\parens{Q^{m}}} 
		\lesssim
		\sum_{i=1}^{j}\parens[\Big]{\parens{\mu\eta}^{i-j-\sigma}\lVert D^{i}u \rVert_{L^{p}\parens{\mA_{\mu}}} + \parens{\mu\eta}^{i-j}\lvert D^{i}u \rvert_{W^{\sigma,p}\parens{\mA_{\mu}}}} + \mu\text{;}
		\]
		\item\label{item:ush_all} for every \( 0 < s < +\infty \),
		\[
		\lVert u^{\sh}_{\mu} - u \rVert_{L^{p}\parens{Q^{m}}}
		\lesssim 
		\lVert u \rVert_{L^{p}\parens{\mA_{\mu}}} + \mu\text{.}
		\]
	\end{enumerate}

	Since \( \upPhi^{\sh}_{\mu} \) is a local diffeomorphism, we know that \( \parens{\upPhi^{\sh}_{\mu}}^{-1}\parens{\mS^{\rtop}_{\mu}} \) is a smooth \( \ell^{\ast} \)\-/dimensional submanifold of \( Q^{m} \).
	Moreover, \( \upPhi^{\rtop}_{\mu} \) and \( \upPhi^{\sh}_{\mu} \) satisfy the assumptions of Lemma~\ref{lemma:Rclas_local_diffeomorphism}.
	This shows that \( u^{\sh}_{\mu} \in \Rsmooth_{m-\floor{sp}-1}\parens{Q^{m};\mN} \) for every \( 0 < \mu < \frac{1}{2} \), and it therefore only remains to prove the \( W^{s,p} \) convergence \( u^{\sm}_{\mu} \to u \) as \( \mu \to 0 \) to conclude the proof. 
	To accomplish this, we verify that the quantities in the right-hand side of~\ref{item:ush_sle1} to~\ref{item:ush_all} converge to \( 0 \) as \( \mu \to 0 \).
	
	We first observe the following estimate on the measure of \( \mA_{\mu} \):
	\begin{equation}
	\label{eq:estimate_mes_Amu}
		\abs{\mA_{\mu}} \lesssim \parens{\mu\eta}^{\ell+1}\text{.}
	\end{equation}
	By Lebesgue's lemma, we deduce that the quantities \( \abs{u}_{W^{s,p}\parens{\mA_{\mu}}} \) and \( \norm{u}_{L^{p}\parens{\mA_{\mu}}} \), that appear on~\ref{item:ush_sle1} and~\ref{item:ush_all} respectively, indeed tend to \( 0 \) as \( \mu \to 0 \).
	Moreover, when \( 0 < s < 1 \), using the fact that \( u \in L^{\infty}\parens{Q^{m}} \) by the compactness of \( \mN \), we have 
	\[
		\lVert u \rVert_{L^{p}\parens{\mA_{\mu}}}
		\lesssim
		\abs{\mA_{\mu}}^{\frac{1}{p}}
		\lesssim
		\parens{\mu\eta}^{\frac{\ell+1}{p}}\text{.}
	\]
	Therefore,
	\[
		\parens{\mu\eta}^{-s}\lVert u \rVert_{L^{p}\parens{\mA_{\mu}}}
		\lesssim
		\parens{\mu\eta}^{\frac{\ell+1-sp}{sp}}\text{,}
	\]
	which converges to \( 0 \) as \( \mu \to 0 \) because of the fact that \( sp < \ell+1 \).
	
	We now consider estimates~\ref{item:ush_sgeq1_integer} and~\ref{item:ush_sgeq1_frac}, when \( s \geq 1 \). 
	Observe that, since \( u \in W^{s,p}\parens{Q^{m}} \cap L^{\infty}\parens{Q^{m}} \), the Gagliardo--Nirenberg inequality implies that \( D^{i}u \in L^{\frac{sp}{i}}\parens{Q^{m}} \) for every \( i \in \set{1,\dots,k} \).
	Hence, Hölder's inequality and~\eqref{eq:estimate_mes_Amu} ensure that 
	\[
		\norm{D^{i}u}_{L^{p}\parens{\mA_{\mu}}}
		\leq
		\abs{\mA_{\mu}}^{\frac{s-i}{sp}} \norm{D^{i}u}_{L^{\frac{sp}{i}}\parens{\mA_{\mu}}}
		\lesssim
		\parens{\mu\eta}^{\frac{\parens{\ell+1}\parens{s-i}}{sp}}\norm{D^{i}u}_{L^{\frac{sp}{i}}\parens{\mA_{\mu}}}\text{.}
	\]
	Therefore, we deduce that 
	\[
		\parens{\mu\eta}^{i-j}\norm{D^{i}u}_{L^{p}\parens{\mA_{\mu}}}
		\lesssim
		\parens{\mu\eta}^{\frac{\parens{\ell+1}\parens{s-i}-\parens{j-i}sp}{sp}}\norm{D^{i}u}_{L^{\frac{sp}{i}}\parens{\mA_{\mu}}}
		\lesssim
		\parens{\mu\eta}^{\frac{\parens{j-i}\parens{\ell+1-sp}}{sp}}\norm{D^{i}u}_{L^{\frac{sp}{i}}\parens{\mA_{\mu}}}\text{.}
	\]
	As \( sp < \ell+1 \), the exponent of \( \mu\eta \) is positive, which implies that the right-hand side converges to \( 0 \) as \( \mu \to 0 \).
	This handles estimate~\ref{item:ush_sgeq1_integer}.
	
	If \( s \geq 1 \) and \( \sigma \neq 0 \), the same reasoning leads to
	\[
		\parens{\mu\eta}^{i-j-\sigma}\norm{D^{i}u}_{L^{p}\parens{\mA_{\mu}}}
		\lesssim
		\parens{\mu\eta}^{\frac{\parens{\sigma+j-i}\parens{\ell+1-sp}}{sp}}\norm{D^{i}u}_{L^{\frac{sp}{i}}\parens{\mA_{\mu}}}\text{,}
	\] 
	which also goes to \( 0 \) as \( \mu \to 0 \).
	Similarly, by interpolation, see~\cite[Lemma~6.1]{Detaille2023}, we find
	\[
		\abs{D^{i}u}_{W^{\sigma,p}\parens{\mA_{\mu}}}
		\lesssim
		\abs{\mA_{\mu}}^{\frac{s-i-\sigma}{sp}}\norm{D^{i}u}^{1-\sigma}_{L^{\frac{sp}{i}}\parens{\mA_{\mu}}}\norm{D^{i+1}u}_{L^{\frac{sp}{i+1}}\parens{Q^{m}}}^{\sigma}
		\lesssim
		\parens{\mu\eta}^{\frac{\parens{s-i-\sigma}\parens{\ell+1}}{sp}}\norm{D^{i}u}^{1-\sigma}_{L^{\frac{sp}{i}}\parens{\mA_{\mu}}}
	\]
	for every \( i \in \set{1,\dots,j-1} \).
	Therefore,
	\[
		\parens{\mu\eta}^{i-j}\abs{D^{i}u}_{W^{\sigma,p}\parens{\mA_{\mu}}}
		\lesssim
		\parens{\mu\eta}^{\frac{\parens{j-i}\parens{\ell+1-sp}}{sp}}\norm{D^{i}u}^{1-\sigma}_{L^{\frac{sp}{i}}\parens{\mA_{\mu}}}\text{,}
	\]
	which once more goes to \( 0 \) as \( \mu \to 0 \).
	This finishes to handle the second term in estimate~\ref{item:ush_sgeq1_frac} when \( i < j \).
	The second term for \( i = j \) is simply \( \abs{D^{j}u}_{W^{\sigma,p}\parens{\mA_{\mu}}} \), which converges to \( 0 \) due to the Lebesgue lemma.
	
	All cases being covered, this finishes to prove that \( u^{\sh}_{\mu}  \to u \) as \( \mu \to 0 \), which concludes the proof of the theorem.
\end{proof}

As a concluding remark, we note that our method uses in a explicit way the fact that the domain is a cube.
However, the argument can be adapted to any domain which has a shape allowing to evacuate crossings as we did for the cube.
For instance, consider the ball with a hole \( B_{2} \setminus B_{1} \subset \bR^{m} \).
One may use a decomposition into cells that are diffeomorphic to cubes and arranged in a radial way, and evacuate crossings between lines along the radial direction to deduce the density of \( \Rsmooth_{1}\parens{B_{2} \setminus B_{1};\mN} \) in \( W^{s,p}\parens{B_{2} \setminus B_{1};\mN} \) when \( \floor{sp} = m-2 \).
The idea of the construction is illustrated on Figure~\ref{fig:radial_uncrossing} in dimension \( m = 2 \), where we have represented the singular set of the map in the radial equivalent of the class \( \Rrig \) to be approximated in red, and the wells used to uncross the singularities in dark blue.
We shall not attempt to present a detailed argument, since it would require to adapt the whole proof of the density of class \( \Rrig \) in a radial version, which would considerably increase the length of this text.
We therefore keep this observation as a remark, and not a theorem with precise statement and proof.
On the other hand, on the same domain, the method does not seem to work to uncross plane singularities for instance, since there is no second direction along which to evacuate the remaining crossings after the first uncrossing step has been performed.

\begin{figure}[ht]
	\centering
	\includegraphics[page=10]{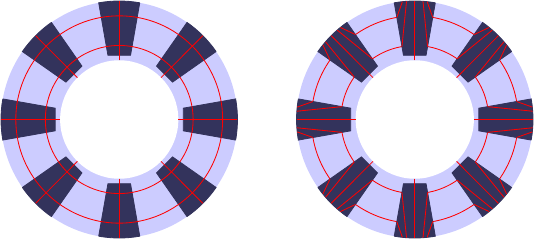}
	\caption{Radial uncrossing procedure}
	\label{fig:radial_uncrossing}
\end{figure}

Nevertheless, this particular situation does not provide a counterexample, since one could extend the map to be approximated inside the hole by homogeneous extension, and then apply the technique we introduced to uncross the singularities of the extended map on \( B_{2} \).
However, such a straightforward extension argument cannot be implemented on a general domain.
Actually, there does not seem to be a direct way to solve the case of a general domain using the technique we introduced for \( Q^{m} \) as such.

It is not clear to us what should be the general situation.
It could be that the class \( \Rsmooth \) is always dense in \( W^{s,p} \), but that the proof for a general domain requires an adaptation of our argument or even a new idea.
It could also be that there are some new obstructions that arise, stemming for instance from the topology of the domain, in the spirit of the work of Hang and Lin~\cite{HangLin2003II}.
This motivates us to conclude on the following open problem.

\begin{openproblem}
\label{openproblem:general_domain}
	Is it true that \( \Rsmooth_{m-\floor{sp}-1}\parens{\Omega;\mN} \) is always dense in \( W^{s,p}\parens{\Omega;\mN} \) for any domain \( \Omega \subset \bR^{m} \) sufficiently smooth?
\end{openproblem}

\bibliographystyle{amsalpha-nodash-init-nosentcase}

\end{document}